\documentclass[reqno,12pt]{amsart}
\usepackage{bbm}
\usepackage[all,pdf]{xy}
\usepackage{epsfig}
\usepackage{amsmath}
\usepackage{amssymb}
\usepackage{amscd}
\usepackage{bm}
\usepackage{graphicx}
\usepackage[colorlinks=true]{hyperref}
\hypersetup{urlcolor=red, citecolor=blue}
\makeatletter
\@namedef{subjclassname@2020}{%
	\textup{2020} Mathematics Subject Classification}

\newtheorem{thm}{Theorem}[section]
\newtheorem{lem}[thm]{Lemma}

\newtheorem{prop}[thm]{Proposition}

\newtheorem{cor}[thm]{Corollary}
\newtheorem{de}[thm]{Definition}

\newtheorem{rem}[thm]{Remark}

\def \N {\mathbb N}

\numberwithin{equation}{section}

\begin{document}
		\author{Wen Huang, Maoru Tan and Leiye Xu} \address {School of Mathematical Sciences, University of Science and Technology of China, Hefei, Anhui, 230026, PR China}
	\email{wenh@mail.ustc.edu.cn}
	\email{lwtanmr@mail.ustc.edu.cn}
    \email{leoasa@mail.ustc.edu.cn}
	
	\title[Logarithmic Sarnak Conjecture]{Almost Countable Spectrum and Logarithmic Sarnak Conjecture}
	
	\thanks{}
	
	
	\begin{abstract} In this paper, we introduce topological dynamical systems with almost countable spectrum. We prove that the Logarithmic Sarnak Conjecture holds for zero-entropy topological dynamical systems whose spectrum is almost countable. This class includes Anzai skew product on $\mathbb{T}^2$ over a rotation of $\mathbb{T}^1$, time-one maps of continuous suspension flows over rotations, systems with finite maximal pattern entropy, and bounded tame systems.
	\end{abstract}
	\maketitle

\section{Introduction}\label{sec-intro}
In this paper, a {\it topological dynamical system} (TDS for short) is a pair $(X, T)$, where
$X$ is a compact metric space endowed with a metric $d$ and $T: X \to X$ is a homeomorphism.
We say a sequence $\xi$ is {\em realized} in $(X,T)$ if there is an $f\in C(X)$ and an $x\in X$
such that $\xi(n) = f(T^nx)$ for any $n\in\mathbb{N}$. A sequence $\xi$ is called {\em deterministic} if it is realized in a TDS
with zero topological entropy.

Let $\bm{\lambda}: \mathbb{Z}\to\{-1,1\}$ be the Liouville function, defined to be $1$ on positive integers with an even number of prime factors (counted with multiplicity) and $-1$ elsewhere. We extend $\bm{\lambda}$ to the integers arbitrarily, for example by setting $\bm{\lambda}(-n)=\bm{\lambda}(n)$ for negative $n\in\mathbb{Z}$ and $\lambda(0)=0$. The M\"{o}bius function $\bm{\mu}$ equals $\bm{\lambda}$ on square-free integers and is $0$ otherwise.
%
Here is the well-known conjecture by Sarnak \cite{Sar}:

\noindent {\bf Sarnak conjecture:}\ {\em
The M\"{o}bius function $\mu$ is linearly asymptotically disjoint from any deterministic sequence $\xi$. That is,
\begin{equation*}
\lim_{N\rightarrow \infty}\frac{1}{N}\sum_{n=1}^N\xi(n)\bm{\mu}(n)=0.
\end{equation*}
}

This is a fundamental yet deeply challenging problem, and partial results for this conjecture have been established for a wide range of dynamical systems (see \cite{DLMR,FKL2018,GT,HWY,KLR,FKL2019,LS15,W17} and references therein).

In 2017, Tao introduced and investigated the following logarithmic version of the Sarnak conjecture \cite{Tao,Tao1} (see also \cite{FHost,HXY,M,TV,TV1}).

\medskip
\noindent {\bf Logarithmic Sarnak conjecture:}\ {\em
For any topological dynamical system $(X,T)$ with zero entropy, any continuous function $f:X\to\mathbb{C}$ and any point $x$ in $X$,
\begin{equation}\label{log-Sarnak}
\lim_{N\rightarrow \infty}\frac{1}{\log N}\sum_{n=1}^N\frac{f(T^nx)\bm{\mu}(n)}{n}=0.
\end{equation}
}

For an overview of recent progress on the logarithmic Sarnak conjecture we refer to the surveys \cite{DLMR,FKL2018,FKL2019}.
In 2018, Frantzikinakis and Host established that the conjecture holds for topological dynamical systems of zero entropy that admit only countably many ergodic measures \cite[Theorem~1.1]{FHost}.

The main theorem of the present paper asserts that the logarithmic Sarnak conjecture holds for topological dynamical systems of zero entropy whose spectrum is \emph{almost countable}.
To make this precise, we first introduce the notion of almost countable spectrum.

Let $(X,T)$ be a TDS and $\mathcal{X}$ be the Borel $\sigma$-algebra of $X$.
We write $\mathcal{M}(X,T)$ for the set of all $T$-invariant Borel probability measures and $\mathcal{M}^{e}(X,T)$ for the subset of ergodic measures.
Fix $\mu\in\mathcal{M}(X,T)$.
A complex number $\lambda$ is called an \emph{eigenvalue} of $(X,\mathcal{X},T,\mu)$ if there exists a non-zero function $f\in L^{2}(X,\mathcal{X},\mu)$ such that $U_{T}f=\lambda f$, where $U_Tf:=f\circ T$ is the Koopman operator; such an $f$ is called an \emph{eigenfunction} associated with $\lambda$.
Since $U_{T}$ is unitary, every eigenvalue satisfies $|\lambda|=1$, hence $\lambda\in\mathbb{T}:=\mathbb{R}/\mathbb{Z}$.
The \emph{spectrum} of $(X,\mathcal{X},T,\mu)$ is defined to be
\[
\operatorname{Spec}(X,\mathcal{X},\mu,T):=\{\lambda\in\mathbb{T}:\lambda\text{ is an eigenvalue of }(X,\mathcal{X},T,\mu)\}.
\]
Because the Hilbert space \(L^{2}(X,\mathcal{X},\mu)\) is separable and eigenfunctions corresponding to distinct eigenvalues are orthogonal, the set \(\operatorname{Spec}(X,\mathcal{X},\mu,T)\) is countable.

We say that an invariant measure $\mu\in\mathcal{M}(X,T)$ has \emph{almost countable spectrum} if there exists a countable subset $C_{\mu}\subset\mathbb{T}$ such that, for the ergodic decomposition $\mu=\int_{\mathcal{M}^{e}(X,T)}m\,d\tau(m)$ of $\mu$, one has
\[
\operatorname{Spec}(X,\mathcal{X},m,T)\subset C_{\mu}\quad\text{for }\tau\text{-a.e.\ }m\in\mathcal{M}^{e}(X,T).
\]
We say that a TDS $(X,T)$ has \emph{almost countable spectrum} if every invariant measure $\mu\in \mathcal{M}(X,T)$ possesses almost countable spectrum in the sense above. It is clear that a TDS with countably many ergodic measures has almost countable spectrum. Hence the following is an extension of Frantzikinakis and Host's result \cite[Theorem 1.1]{FHost}.

\begin{thm}\label{thm-A}
The logarithmic Sarnak conjecture holds for any TDS with zero entropy and almost countable spectrum.
\end{thm}

In what follows, we present several illustrative examples as applications of Theorem \ref{thm-A}.
First, we consider a group extension of a zero-entropy TDS with countably many ergodic measures.

For a TDS $(X,T)$, let $\operatorname{Aut}(X)=\operatorname{Aut}(X,T)$ denote the group of all automorphisms of the system, i.e., the collection of all homeomorphisms $\phi\colon X\to X$ satisfying $\phi\circ T=T\circ\phi$.
Equipped with the uniform topology, $\operatorname{Aut}(X)$ is a Polish group.
If $K$ is a compact subgroup of $\operatorname{Aut}(X)$, then the map $x\mapsto Kx$ defines a factor map $(X,T)\to (Y,R)$ with $Y=X/K$ and relation
$R_{\pi}=\{(x,kx):x\in X,\;k\in K\}$.
Such an extension is called a \emph{group extension}.

\begin{thm}\label{LSC-group-extension}
Let $\pi\colon(X,T)\to(Y,R)$ be a group extension between two TDSs.
If $(Y,R)$ has zero entropy and only countably many ergodic measures, then the logarithmic Sarnak conjecture holds for $(X,T)$.
\end{thm}

A TDS is called \emph{minimal} if it has no non-trivial closed invariant subset. A TDS $(X,T)$ is \emph{distal} if $\inf_{n\in\mathbb Z}d(T^nx,T^ny)>0$ for all $x\ne y$ in $X$.
In a distal TDS every point $x\in X$ is minimal: the closure of the orbit $\{T^nx:n\in\mathbb Z\}$ is a minimal set.
It is well known that a distal TDS has zero entropy \cite{P68}.
The most typical distal TDS are the \emph{isometric} (also called \emph{equicontinuous}) ones: we say $(X,T)$ is \emph{isometric} if there exists a compatible metric $d$ on $X$ such that every $T^n:X\to X$ is an isometry, i.e.
\(d(T^nx,T^ny)=d(x,y)\) for all $n\in\mathbb{Z}$ and $x,y\in X$.

In 1963 Furstenberg introduced the notions of isometric extension and quasi-isometric TDS (see \cite[Definitions 2.1 and 2.5]{F63}) and proved a remarkable theorem (\cite[Theorem 2.4]{F63}):
every minimal distal TDS is quasi-isometric, i.e. it can be obtained by a (possibly transfinite) tower of isometric extensions starting from the one-point system; the tower may have any countable ordinal $\eta$ height. In Section 13 of \cite{F63} the \emph{order} of a minimal distal TDS is defined as follows.
In the construction of a quasi-isometric TDS there appears an ordinal $\eta$ that counts the number of successive isometric extensions needed to reach the given system from the trivial one.
This ordinal is not unique: if $T:\mathbb T\to\mathbb T$ is an irrational rotation, then $\phi(z)=z^2$ defines a factor $(\mathbb T,T^2)$ of $(\mathbb T,T)$;
$(\mathbb T,T^2)$ is an isometric extension of the trivial system and $(\mathbb T,T)$ is an isometric extension of $(\mathbb T,T^2)$, so the ordinal could be 1, 2, or any finite integer.
Nevertheless, the least such ordinal is well defined.
For a minimal distal TDS $(X,T)$ its \emph{order} is the smallest ordinal $\eta$ for which $(X,T)$ appears as the top of a quasi-isometric tower of height $\eta$ (see \cite[Definitions 2.4 and 2.5]{F63}).

Clearly the order of a non-trivial minimal isometric TDS is $1$. The Halmos-von Neumann theorem tells us that a minimal isometric TDS is \emph{uniquely ergodic} (it carries a single invariant measure) and, up to topological conjugacy, is completely determined by its eigenvalues: the corresponding countable subgroup of $\mathbb T$.
Concretely, it is topologically conjugate to a minimal rotation on a compact abelian metrisable group (see e.g.\ \cite[Theorem 5.18]{W82}).
By Theorem~\ref{LSC-group-extension}, the Logarithmic Sarnak Conjecture holds for every group extension of a minimal isometric TDS.

Next, observe that every isometric extension between minimal systems can be lifted to a group extension (see \cite[Definition 2.1]{F63}).
Hence a minimal distal TDS of order $2$ is a factor of a group extension of a minimal isometric TDS.
Since the logarithmic M\"{o}bius disjointness property \eqref{log-Sarnak} descends along factor maps, we obtain

\begin{thm}\label{LSC-distal-order2}
Let $(X,T)$ be a distal TDS. If the order of every minimal subsystem of $(X,T)$ is at most $2$, then the logarithmic Sarnak conjecture holds for $(X,T)$.
\end{thm}

Let $\mathbb{T}=\mathbb{R}/\mathbb{Z}$ denote the circle and $d\in \mathbb{N}$.
The \emph{Anzai skew products} $(\mathbb{T}^{d+1},T_{\alpha,\phi})$, where $\alpha\in\mathbb{R}^d$, $\phi\colon\mathbb{T}^d\to\mathbb{T}$ is continuous and
\[
T_{\alpha,\phi}(x,y):=(x+\alpha,y+\phi(x))\qquad\text{for all }(x,y)\in\mathbb{T}^{d+1},
\]
will often be denoted simply by $T_{\alpha,\phi}$.
As a direct application of Theorem~\ref{LSC-distal-order2} we obtain

\begin{cor}
For any  $\alpha\in\mathbb{R}^d$ and any continuous $\phi\colon\mathbb{T}^d\to\mathbb{T}$, the logarithmic Sarnak conjecture holds for the Anzai skew product $T_{\alpha,\phi}$.
\end{cor}

When $d=1$, we remark that Liu and Sarnak~\cite{LS15} showed that if $\alpha\in \mathbb{R}$ is rational, then the Sarnak conjecture holds for $T_{\alpha,\phi}$.
The first M\"obius-disjointness result for \emph{all} $\alpha$ was established by Liu and Sarnak~\cite{LS15}, who proved the conjecture for $T_{\alpha,\phi}$ with $\phi$ analytic and satisfying the technical condition $|\hat{\phi}(m)|\gg e^{-\tau|m|}$ for some $\tau>0$.
A refinement was obtained by Wang~\cite{W17}, who removed the lower-bound requirement on Fourier coefficients and verified the conjecture for analytic $\phi$. Huang, Wang and Ye~\cite{HWY} later improved this to cover every $\phi\in C^{\infty}$.
In 2021, Kanigowski, Lema\'nczyk and Radziwi\l\l~\cite{KLR} confirmed the conjecture for $\phi\in C^{2+\varepsilon}$ with $\hat{\phi}(0)=0$, where $\varepsilon>0$ is arbitrary.
Finally, de Faveri~\cite{DF22} improved the result to the case $\phi\in C^{1+\varepsilon}$.
We also note that, by \cite[Remark~2.5.7]{KL15} or \cite[Corollary~2.6]{W17}, if $\phi$ is Lipschitz continuous and not homotopically trivial, then the Sarnak conjecture holds for $T_{\alpha,\phi}$. Wei \cite{W}
showed that the logarithmic Sarnak conjecture holds for $T_{\alpha,\phi}$ when $\phi$ is Lipschitz continuous.

Next we consider time-one maps of continuous suspension flows over zero entropy TDSs with countably many ergodic measures. Let us recall some basic facts about suspension flow that can be found in the book by Parry and Pollicott \cite{PP}. Let $(X,T)$ be a TDS. Let $r\colon X\to(0,\infty)$ be a strictly positive, continuous function (referred to as a \emph{roof function} on $X$), and consider the space
\[
X_r=\bigl\{(x,s):x\in X,\;0\le s\le r(x)\bigr\}
\]
with the points $(x,r(x))$ and $(Tx,0)$ identified for each $x\in X$.

The \emph{suspension flow} $\Phi=(\varphi_t)_{t\in\mathbb R}$ over $(X,T)$ with roof function $r(x)$ is the flow on $X_r$ defined, for $(x,s)\in X_r$ and $t\in\mathbb R$, by
\[
\varphi_t(x,s)=\bigl(T^nx,\;s+t-r_n(x)\bigr),
\]
where
\(
r_n(x)=
\begin{cases}
\sum_{i=0}^{n-1}r(T^ix), & n>0\\
0, & n=0\\
\sum_{i=1}^{|n|}r(T^{-i}x), & n<0
\end{cases}
\)
and $n\in\mathbb{Z}$ is the unique integer satisfying
\(
r_n(x)\le s+t<r_{n+1}(x).
\)
Endowing $X_r$ with the Bowen-Walters metric makes $\Phi$ a continuous $\mathbb R$-action (see \cite[Section 4]{BW}).

\begin{thm}\label{LSC-susp}Let $(X,T)$ be a TDS with zero entropy and countable many ergodic measures. For any continuous function $r:X\to (0,+\infty)$, the logarithmic Sarnak conjecture holds for $(X_r,\varphi_1)$.
\end{thm}

We now turn to TDSs whose maximal pattern entropy is finite.
Systems with finite maximal pattern entropy provide further examples possessing almost countable spectrum.
Motivated by the notion of maximal pattern complexity introduced by Kamae and Zamboni \cite{KZ02} in 2002, Huang and Ye \cite{HY09} defined the maximal pattern entropy for an invariant measure and for a TDS alike.
Maximal pattern entropy is intimately related to sequence entropy, an important invariant for measuring the complexity of zero-entropy systems.
Topological sequence entropy and measure-theoretic sequence entropy were introduced by Goodman \cite{G} and Ku\v{s}nirenko \cite{Ku}, respectively.
As shown in \cite{HY09}:
\begin{enumerate}
\item[(i)] maximal pattern entropy coincides with the supremum of sequence entropies taken over all sequences;

\item[(ii)] the maximal pattern entropy of a TDS equals $\log k$ for some integer $k$ or is $+\infty$, where $k$ is the maximal length of a so-called intrinsic sequence-entropy tuple;

\item[(iii)] an invariant measure of a TDS has zero maximal pattern entropy if and only if it has discrete spectrum.
\end{enumerate}

\begin{thm}\label{LSC-bounded-m-MPE} Let $(X,T)$ be a TDS. If every invariant probability measure of $(X,T)$ has finite maximal pattern entropy, then  the logarithmic Sarnak conjecture holds for $(X,T)$.
\end{thm}

The precise definition of maximal pattern entropy and further properties are given in Subsection~\ref{subsect-mpe}. In particular, Theorem~2.6 of \cite{HY09} (see also Theorem~\ref{MPE-p}~(6) in Subsection~\ref{subsect-mpe}) shows that for a TDS $(X,T)$ the maximal pattern entropy of every invariant measure is less than or equal to the maximal pattern entropy of $(X,T)$. Therefore, as a direct corollary of Theorem~\ref{LSC-bounded-MPE} and Theorem~2.6 of \cite{HY09}, we obtain that the logarithmic Sarnak conjecture holds for any TDS with finite maximal pattern entropy.

By using the notion of IT-tuple introduced by Kerr and Li \cite{KL07} (see also \cite{Huang}), we can strengthen this result further. A TDS \((X,T)\) is said to \emph{have no \(K\)-IT-tuple} (\(K\ge 2\)) if for every \(K\)-tuple of pairwise disjoint closed subsets \((U_1,U_2,\dots,U_K)\) of $X$ there is no infinite independence set; that is, for every infinite set \(S\subset\mathbb N\) there exists \(a\in\{1,2,\dots,K\}^S\) such that
\(\bigcap_{n\in S}T^{-n}U_{a_n}=\emptyset\). A topological dynamical system is called \emph{tame} if its enveloping semigroup is separable and Fr\'echet \cite{G3}. Kerr and Li \cite{KL07} proved that a topological dynamical system is tame if and only if it has no \(2\)-IT-tuple. Huang \cite{Huang} showed that every invariant measure of a tame system has discrete spectrum. Huang, Wang and Zhang \cite{HWZ} obtained that the Sarnak conjecture holds for any TDS whose invariant measures all have discrete spectrum. More generally, Huang, Wang and Ye \cite{HWY} introduced \emph{mean measure complexity} and proved the Sarnak conjecture whenever every invariant measure has sub-polynomial mean measure complexity. Huang and Xu \cite{HX} verified sub-polynomial topological complexity for suspension flows over an irrational rotation with a \(C^\infty\) roof function; combining their result with \cite[Theorem 1.1]{HWY} yields Sarnak's conjecture for the time-one map of these flows. Huang, Xu and Ye \cite{HXY} further reduced the logarithmic Sarnak conjecture to \(\{0,1\}\)-symbolic systems whose mean complexity is polynomial.

We call a topological dynamical system \((X,T)\) \emph{bounded tame} if there exists \(K\ge 2\) such that \((X,T)\) has no \(K\)-IT-tuple. By Corollary 1.5 of \cite{LXZ}, every invariant measure of a bounded tame TDS has finite maximal pattern entropy; consequently the logarithmic Sarnak conjecture holds for every bounded tame TDS. Summing up, as a direct corollary of Theorem~\ref{LSC-bounded-MPE}, Theorem~2.6 of \cite{HY09} and Corollary~1.5 of \cite{LXZ}, we have

\begin{cor}\label{LSC-bounded-MPE}
Let $(X,T)$ be a TDS having finite maximal pattern entropy or being bounded tame. Then the logarithmic Sarnak conjecture holds for $(X,T)$.
\end{cor}

The structure of the paper is as follows. In Section \ref{sec-pre}, we recall some basic notions and results. In Section \ref{sec-dis}, we prove a disjointness result, which will be used in the proof of our main theorem. In Section \ref{sec-proof-thm-a}, we prove Theorem \ref{thm-A}. In Section~\ref{sec-example} we exhibit several systems whose spectrum is almost countable.
Our principal aim, however, is to deduce Theorems~\ref{LSC-group-extension}, \ref{LSC-susp} and \ref{LSC-bounded-m-MPE} from Theorem~\ref{thm-A}.

	\section{Preliminaries}\label{sec-pre}		
Throughout this paper, we denote by $\N$ the set of natural numbers. Denote by $e(t)=e^{2\pi it}$ for $\mathbb{R}$. Firstly we review some basic notions and fundamental properties of dynamical systems.

\subsection{Disintegration of Borel probability measures}\label{subs-Disin}
Let $X$ be a dense Borel subset of a compact metric space, endowed with a probability measure defined on the restriction of the Borel $\sigma$-algebra $\mathcal{X}$ to $X$.
The resulting probability space $(X,\mathcal{X},\mu)$ is called a \emph{Borel probability space}.

For a Borel probability space $(X,\mathcal{X},\mu)$ and a sub-$\sigma$-algebra $\mathcal{C}\subseteq\mathcal{X}$, it is well known (see, for example, \cite[Theorem~5.8]{F81} or \cite[Theorem~5.14]{EW11}) that $\mu$ can be disintegrated over $\mathcal{C}$ as
\(
\mu=\int_{X}\mu^{\mathcal{C}}_{x}\,d\mu(x),
\)
referred to as \emph{conditional measures}, with the following properties:
\begin{enumerate}
\item $\mu^{\mathcal{C}}_{x}$ is a probability measure on $(X,\mathcal{X})$ satisfying
\begin{equation*}
\mathbb{E}_{\mu}(f\mid\mathcal{C})(x)=\int_{X}f\,d\mu^{\mathcal{C}}_{x}
\quad\text{for $\mu$-a.e.\ }x\in X
\end{equation*}
for every $f\in L^{1}(X,\mathcal{X},\mu)$, where $\mathbb{E}_{\mu}(f\mid\mathcal{C})$ denotes the conditional expectation of $f$ given $\mathcal{C}$ with respect to $\mu$.

In other words, for every $f\in L^{1}(X,\mathcal{X},\mu)$ the integral $\int f(y)\,d\mu^{\mathcal{C}}_{x}(y)$ exists for all $x$ in a $\mathcal{C}$-conull set, the map
\(
x\mapsto\int f(y)\,d\mu^{\mathcal{C}}_{x}(y)
\)
is $\mathcal{C}$-measurable on this set, and
\[
\int_{A}\!\Bigl(\int f(y)\,d\mu^{\mathcal{C}}_{x}(y)\Bigr)\,d\mu(x)
=\int_{A}f\,d\mu
\quad\text{for all }A\in\mathcal{C}.
\]

\item For $\mu$-a.e.\ $x\in X$,
$\mu^{\mathcal{C}}_{x}=\mu^{\mathcal{C}}_{x'}$
for $\mu^{\mathcal{C}}_{x}$-a.e. $x'\in X$.
\end{enumerate}
Meanwhile, for a Borel probability space $(X,\mathcal{X},\mu)$ and a sub-$\sigma$-algebra $\mathcal{C}\subseteq\mathcal{X}$, we can always find a Borel probability space $(Y,\mathcal{Y},\nu)$ and a measurable map $$\pi\colon(X,\mathcal{X})\to(Y,\mathcal{Y})$$ such that $\pi_{*}\mu=\nu$ and $\mathcal{C}$ coincides with $\pi^{-1}(\mathcal{Y})$ modulo $\mu$-null sets (see, e.g., Lemma~5.17 and Corollary~5.22 in \cite{EW11}).

Next, let $\pi\colon(X,\mathcal{X},\mu)\to(Y,\mathcal{Y},\nu)$ be a measurable map between two Borel probability spaces satisfying $\pi_{*}\mu=\nu$.
Then $\pi^{-1}(\mathcal{Y})$ is a sub-$\sigma$-algebra of $\mathcal{X}$, and there exists a natural disintegration
\(
\mu=\int_{Y}\mu_{y}\,d\nu(y),
\)
called the \emph{disintegration of $\mu$ over $\pi$}, such that
\(\mu_{\pi(x)}=\mu^{\pi^{-1}(\mathcal{Y})}_{x}\) for $\mu$-a.e. $x\in X$.
We shall freely use whichever of the two representations of measure disintegration (over a sub-$\sigma$-algebra or over a factor map) is more convenient in the given context.

\subsection{Measure preserving systems}
Let $X$ be a compact metric space and let $\mathcal{M}(X)$ denote the set of all Borel probability measures on $X$.
For a TDS $(X,T)$, recall that $\mathcal{M}(X,T)$ and $\mathcal{M}^{e}(X,T)$ for the sets of all $T$-invariant Borel probability measures and all ergodic measures of $(X,T)$, respectively.
For $\mu\in\mathcal{M}^{e}(X,T)$, let
\[
\text{Gen}(\mu):=\Bigl\{x\in X:
\lim_{N\to\infty}\frac{1}{N}\sum_{i=0}^{N-1}\delta_{T^{i}x}=\mu
\text{ in the weak$^{*}$-topology}\Bigr\}.
\]
Clearly $\text{Gen}(\mu)\in\mathcal{X}$, and Birkhoff's ergodic theorem gives $\mu\bigl(\text{Gen}(\mu)\bigr)=1$.

For a given $\mu\in\mathcal{M}^{e}(X,T)$, there exist a Borel probability space $(\Omega,\mathcal{O},\xi)$ and a measurable map $\omega\mapsto\mu_{\omega}$ from $\Omega$ to $\mathcal{M}^{e}(X,T)$ such that
\[
\int_{X}f\,d\mu=\int_{\Omega}\!\Bigl(\int_{X}f\,d\mu_{\omega}\Bigr)\,d\xi(\omega)
\qquad\text{for every }f\in L^{1}(X,\mathcal{X},\mu).
\]

We rewrite this as $\mu=\int_{\Omega}\mu_{\omega}\,d\xi(\omega)$ and call it the \emph{ergodic decomposition of $\mu$}.
In fact, letting $\mathcal{I}_{\mu}(T)=\{B\in\mathcal{X}:T^{-1}B=B\}$, the ergodic decomposition of $\mu$ is exactly the disintegration of $\mu$ over $\mathcal{I}_{\mu}(T)$ (cf.\ Subsection~\ref{subs-Disin} with $\mathcal{C}=\mathcal{I}_{\mu}(T)$). We call the systems $(X,\mathcal{X},\mu_{\omega},T)$, $\omega\in \Omega$, \emph{the ergodic components} of $(X,\mathcal{X},\mu,T)$.

We remark that $\mathcal{M}^{e}(X,T)$ is a $G_{\delta}$-subset of the compact metric space $\mathcal{M}(X,T)$, hence it is a Borel subset of $\mathcal{M}(X,T)$.
Consequently we may take $\Omega=\mathcal{M}^{e}(X,T)$, $\mathcal{O}$ to be the Borel $\sigma$-algebra $\mathcal{B}_{\mathcal{M}^{e}(X,T)}$ restricted to $\mathcal{M}^{e}(X,T)$, and $\xi=\tau$ to be a probability measure on $(\mathcal{M}^{e}(X,T),\mathcal{B}_{\mathcal{M}^{e}(X,T)})$.
Thus the ergodic decomposition of $\mu$ can also be written as
\(\mu=\int_{\mathcal{M}^{e}(X,T)}m\,d\tau(m)\).

Throughout this paper, a \emph{measure preserving system}, or simply \emph{a system}, is a quadruple $(X,\mathcal{X},\mu,T)$, where $(X,T)$ is a topological dynamical system, $\mathcal{X}$ is the Borel $\sigma$-algebra of $X$, $\mu\in\mathcal{M}(X,T)$. A \emph{factor map}  from a system $(X,\mathcal{X},\mu,T)$ to a system $(Y,\mathcal{Y},\nu,S)$ is a measurable map $\pi\colon X\to Y$, such that
\(\pi_{*}\mu:=\mu\circ\pi^{-1}=\nu\)
and with $S\circ\pi(x)=\pi\circ T(x)$ valid $\mu$-almost
everywhere.
When such a map exists we say that $(Y,\mathcal{Y},\nu,S)$ is a \emph{factor} of $(X,\mathcal{X},\mu,T)$.
If the factor map $\pi$ is invertible (i.e., there exists a factor map $Y\rightarrow X$, written $\pi^{-1}$, with $\pi^{-1}\circ \pi=id_X$ valid
$\mu$-almost everywhere), then the two systems are said to be \emph{measure-theoretically isomorphic}, and we write
\((X,\mathcal{X},\mu,T)\cong(Y,\mathcal{Y},\nu,S)\).

Let $(Y,\mathcal{Y},\nu,S)$ be a factor of $(X,\mathcal{X},\mu,T)$.
A factor is characterised (modulo isomorphism) by $\pi^{-1}(\mathcal{Y})$, which is a $T$-invariant sub-$\sigma$-algebra of $\mathcal{X}$; conversely, every $T$-invariant sub-$\sigma$-algebra of $\mathcal{X}$ defines a factor (see for example, \cite[Theorem 5.15]{F81} or \cite[Theorem 6.5]{EW11}).
By a classical abuse of notation we identify the $\sigma$-algebra $\mathcal{Y}$ with its inverse image $\pi^{-1}(\mathcal{Y})$; in other words, we regard $\mathcal{Y}$ as a sub-$\sigma$-algebra of $\mathcal{X}$.
Consequently $L^{2}(Y,\mathcal{Y},\nu)$ is viewed as a closed subspace of $L^{2}(X,\mathcal{X},\mu)$.

\begin{de}\label{de-fi-one} Let $\pi: (X,\mathcal{X},\mu,T)\rightarrow (Y,\mathcal{Y},\nu,S)$ be a factor map between two measure preserving systems
and $\mu=\int_{Y}\mu_yd\nu(y)$ be the disintegration of $\mu$ over $\pi$. We say the factor map $\pi$ is
\emph{almost everywhere finite-to-one} if $\mu_y$ support on  a finite set for $\nu$-a.e. $y\in Y$.
\end{de}

Next, we introduce the notion of disjointness for two measure-preserving systems, first formulated by Furstenberg \cite{F}.
Let $(X,\mathcal{X},\mu,T)$ and $(Y,\mathcal{Y},\nu,R)$ be two systems.
A \emph{joining} of these systems is a $(T\!\times\!R)$-invariant probability measure $\lambda$ on the Cartesian product $(X\!\times\!Y,\mathcal{X}\!\times\!\mathcal{Y})$ whose marginals are $\mu$ and $\nu$, respectively.
The two systems are said to be \emph{disjoint} if the product measure $\mu\!\times\!\nu$ is the only joining.
Furstenberg \cite{F} proved that a measure-preserving system has zero entropy if and only if it is disjoint from every Bernoulli system.
Here a \emph{Bernoulli system} has the form $(Y^{\mathbb{Z}},\mathcal{Y}^{\mathbb{Z}},\nu^{\mathbb{Z}},\sigma)$, where $(Y,\mathcal{Y},\nu)$ is a Borel probability space, $\sigma$ is the left shift on $Y^{\mathbb{Z}}$, $\mathcal{Y}^{\mathbb{Z}}$ is the product $\sigma$-algebra, and $\nu^{\mathbb{Z}}$ is the product measure.

\subsection{Maximal pattern entropy}\label{subsect-mpe} In this subsection we review sequence entropy \cite{G,Ku} and the maximal pattern entropy \cite{HY09}.
Let $X$ be a compact metric space and let $\mathcal{X}$ denote Borel $\sigma$-algebra, that is, the collection of all Borel subsets of $X$; measurability will always refer to $\mathcal{X}$. In this article a \emph{cover} of $X$ is a finite family of Borel subsets whose union is $X$, and a \emph{partition} of $X$ is a cover whose elements are pairwise disjoint. We write  $\mathcal{P}_X$ for the set of all partitions of $X$, $\mathcal{C}_X$ for the set of all covers of $X$,
$\mathcal{C}_X^o$ for the set of all open covers of $X$. Given two covers $\mathcal{U},\mathcal{V}\in\mathcal{C}_X$, we say $\mathcal{U}$ is \emph{finer} than $\mathcal{V}$ (write $\mathcal{U}\succeq\mathcal{V}$) if each element of $\mathcal{U}$ is contained in some element of $\mathcal{V}$.
Define $\mathcal{U}\vee\mathcal{V}=\{U\cap V:U\in\mathcal{U},\,V\in\mathcal{V}\}$.

Let $(X,T)$ be a topological dynamical system. Denote by $\mathcal{S}$ the set of all strictly increasing sequences of $\mathbb{Z}_+:=\{0\}\cup \mathbb{N}$.
For  $\mathcal{A}=(n_{i})_{i=1}^{\infty}\in \mathcal{S}$ and $\mathcal{U}\in\mathcal{C}^{o}_{X}$, the \emph{topological sequence entropy of $\mathcal{U}$} with respect to $(X,T)$ along $\mathcal{A}$ is defined by
\[
h_{\mathrm{top}}^{\mathcal{A}}(T,\mathcal{U})
=\limsup_{N\to\infty}\frac{1}{N}\log N\!\left(\bigvee_{i=1}^{N}T^{-n_{i}}\mathcal{U}\right),
\]
where $N\!\left(\bigvee_{i=1}^{N}T^{-n_{i}}\mathcal{U}\right)$ denotes the minimal cardinality of a sub-cover of $\bigvee_{i=1}^{N}T^{-n_{i}}\mathcal{U}$.
The \emph{topological sequence entropy} of $(X,T)$ along $\mathcal{A}$ is then
\[
h_{\mathrm{top}}^{\mathcal{A}}(X,T)
=\sup_{\mathcal{U}\in\mathcal{C}^{o}_{X}}h_{\mathrm{top}}^{\mathcal{A}}(T,\mathcal{U}).
\]
If $\mathcal{A}=\mathbb{N}$ we recover standard topological entropy of $(X,T)$, which we denote simply by $h_{\mathrm{top}}(X,T)$.

For $\mu\in\mathcal{M}(X)$, let $\mathcal{P}^{\mu}_{X}$ denote the set of all finite measurable partitions of $X$.
Given a partition $\alpha\in\mathcal{P}^{\mu}_{X}$ and a sub-$\sigma$-algebra $\mathcal{D}\subseteq\mathcal{X}$, set
\[
H_{\mu}(\alpha\mid\mathcal{D})
=\sum_{A\in\alpha}\int_{X}-\mathbb{E}_{\mu}(1_{A}\mid\mathcal{D})\log\mathbb{E}_{\mu}(1_{A}|\mathcal{D})\,d\mu.
\]
It is standard that $H_{\mu}(\alpha| \mathcal{D})$ increases with $\alpha$ and decreases with $\mathcal{D}$.
Writing $\mathcal{N}=\{\emptyset,X\}$, we define
\(H_{\mu}(\alpha)=H_{\mu}(\alpha|\mathcal{N})
=-\sum_{A\in\alpha}\mu(A)\log\mu(A)\).

Now let $(X,\mathcal{X},\mu,T)$ be a measure-preserving system and let $\mathcal{A}=(n_{i})_{i=1}^{\infty}\in \mathcal{S}$.
The \emph{sequence entropy of a partition $\alpha\in\mathcal{P}^{\mu}_{X}$} with respect to $(X,\mathcal{X},\mu,T)$ along $\mathcal{A}$ is
\[
h_{\mu}^{\mathcal{A}}(T,\alpha)
=\limsup_{N\to\infty}\frac{1}{N}H_{\mu}\!\left(\bigvee_{i=1}^{N}T^{-n_{i}}\alpha\right),
\]
and \emph{the sequence entropy of the system along $\mathcal{A}$} is
$h_{\mu}^{\mathcal{A}}(T)
=\sup_{\alpha\in\mathcal{P}^{\mu}_{X}}h_{\mu}^{\mathcal{A}}(T,\alpha)$.
As in
the topological case, when $\mathcal{A}=\mathbb{N}$, $h_{\mu}^{\mathbb{N}}(T)$ coincides with the usual \emph{measure-theoretic entropy} $h_{\mu}(T)$.
For the classical
theory of measure-theoretical entropy and classical theory of topological entropy can found in \cite{DGS,G03,P04,W82}.

In \cite{HY09}, Huang and Ye introduced the notion of maximal pattern entropy. For a TDS $(X,T)$, $n\in\mathbb N$, and $\mathcal{U}\in \mathcal{C}_X^o$, set
\[
p_{X,\mathcal U}^*(n)=\max_{(t_1<t_2<\dots<t_n)\in\mathbb Z_+^n}
N\!\biggl(\bigvee_{i=1}^n T^{-t_i}\mathcal U\biggr).
\]
\emph{The maximal pattern entropy} of $T$ with respect to $\mathcal U$ is defined by
\[
h_{\text{top}}^*(T,\mathcal U)=\limsup_{n\to+\infty}\frac1n\log p_{X,\mathcal U}^*(n).
\]
It is easy to see that $\{\log p_{X,\mathcal U}^*(n)\}_{n=1}^\infty$ is a sub-additive sequence; hence
\[
h_{\text{top}}^*(T,\mathcal U)=\lim_{n\to+\infty}\frac1n\log p_{X,\mathcal U}^*(n)
=\inf_{n\ge 1}\frac1n\log p_{X,\mathcal U}^*(n).
\]
\emph{The maximal pattern entropy} of $(X,T)$ is
$h_{\text{top}}^*(T)=\sup_{\mathcal U}h_{\text{top}}^*(T,\mathcal U)$,
where the supremum is taken over all finite open covers of $X$.

Analogously, given a system\ $(X,\mathcal{X},\mu,T)$ and $\alpha\in\mathcal{P}^{\mu}_{X}$, we define
\[
p_{X,\alpha,\mu}^*(n)=\max_{(t_1<t_2<\dots<t_n)\in\mathbb Z_+^n}\sum_{A\in\bigvee_{i=1}^n T^{-t_i}\alpha}-\nu(A)\log\nu(A)
\]
and \emph{the maximal pattern entropy of $T$ with respect to $\alpha$} by
\[
h_\mu^*(T,\alpha)=\limsup_{n\to+\infty}\frac1n p_{X,\alpha,\mu}^*(n).
\]
It is easy to see that $\{p_{X,\alpha,\mu}^*(n)\}_{n=1}^\infty$ is a sub-additive sequence; hence
\[
h_\mu^*(T,\alpha)=\lim_{n\to+\infty}\frac1n p_{Y,\alpha,\nu}(n)
=\inf_{n\ge 1}\frac1n p_{X,\alpha,\mu}^*(n).
\]
\emph{The maximal pattern entropy of $T$} is
$h_\mu^*(T)=\sup_\alpha h_\mu^*(T,\alpha)$,
where the supremum is taken over all finite measurable partitions.

The following result is taken from Theorems 2.2, 2.3, 2.6, 4.4 and Corollary 3.9 in \cite{HY09}.
\begin{thm}\label{MPE-p}
Let $(X,T)$ be a TDS and $\mu\in \mathcal{M}(X,T)$.
\begin{enumerate}
\item[\rm(1)]
For every $\mathcal{U}\in \mathcal{C}_X^o$, $h_{\mathrm{top}}^{*}(T,\mathcal{U})=\sup_{\mathcal{A}\in \mathcal{S}}h_{\mathrm{top}}^{\mathcal{A}}(T,\mathcal{U})$,
and there exists $\mathcal{A}\in\mathcal{S}$ such that
$h_{\mathrm{top}}^{*}(T,\mathcal{U})=h_{\mathrm{top}}^{\mathcal{A}}(T,\mathcal{U})$.
Moreover,
$h_{\mathrm{top}}^{*}(T)=\sup_{\mathcal{A}\in \mathcal{S}}h_{\mathrm{top}}^{\mathcal{A}}(T)$
and $h_{\mathrm{top}}^{*}(T)=+\infty$ or $\log k$ for some $k\in \mathbb{N}$.

\item[\rm(2)]
For every $\alpha\in \mathcal{P}_X^\mu$,
$h_{\mu}^{*}(T,\alpha)=\sup_{\mathcal{A}\in \mathcal{S}}h_{\mu}^{\mathcal{A}}(T,\alpha)$,
and there exists $\mathcal{A}\in\mathcal{S}$ such that
$h_{\mu}^{*}(T,\alpha)=h_{\mu}^{\mathcal{A}}(T,\alpha)$.
Moreover, $h_{\mu}^{*}(T)=\sup_{\mathcal{A}\in \mathcal{S}}h_{\mu}^{\mathcal{A}}(T)$
and if in additionally, $\mu$ is ergodic, then
$h_{\mu}^{*}(T)=+\infty$ or $\log k$ for some $k\in \mathbb{N}$.

\item[\rm(3)] if $h_{top}(T)>0$ then $h_{top}^*(T)=+\infty$.

\item[\rm(4)] if $h_{\mu}(T)>0$ then $h_{\mu}^*(T)=+\infty$.

\item[\rm(5)] for every $\mathcal{A}\in\mathcal{S}$,
\(h^{\mathcal{A}}_{\mu}(T) \le h^{\mathcal{A}}_{top}(T)\).
Moreover,
$h^{*}_{\mu}(T) \le h^{*}_{\mathrm{top}}(T)$.
\end{enumerate}
\end{thm}
By Theorem \ref{MPE-p}~(1) and (3), if a TDS $(X,T)$ has   finite
maximal pattern entropy, then $(X,T)$ has zero entropy and $h_{top}^*(T)=\log k$ for some $k\in \mathbb{N}$.

\subsection{Kronecker factor and Discrete spectrum}
Let $(X,\mathcal{X},\mu,T)$ be a system.
In the introduction we defined the spectrum $\operatorname{Spec}(X,\mathcal{X},\mu,T)$.
Recall that a complex number $\lambda$ is an \emph{eigenvalue} of $(X,\mathcal{X},T,\mu)$ if there exists a non-zero function $f\in L^{2}(X,\mathcal{X},\mu)$ such that the Koopman operator satisfies $U_{T}f=\lambda f$, where $U_Tf=f\circ T$; such an $f$ is called an \emph{eigenfunction} associated with $\lambda$.
Since $U_{T}$ is unitary, every eigenvalue satisfies $|\lambda|=1$, hence $\lambda\in\mathbb{T}$.
The \emph{spectrum} of $(X,\mathcal{X},T,\mu)$ is
\[
\operatorname{Spec}(X,\mathcal{X},\mu,T):=\{\lambda\in\mathbb{T}:\lambda\text{ is an eigenvalue of }(X,\mathcal{X},T,\mu)\}.
\]
Because $L^{2}(X,\mathcal{X},\mu)$ is separable and eigenfunctions corresponding to distinct eigenvalues are orthogonal, $\operatorname{Spec}(X,\mathcal{X},\mu,T)$ is at most countable.

Next we introduce the rational and irrational parts of the spectrum, referred to respectively as the \emph{rational spectrum} and the \emph{irrational spectrum}:
\[
\begin{aligned}
\operatorname{Spec}_{\mathrm{rat}}(X,\mathcal{X},\mu,T)
&=\{\lambda\in\operatorname{Spec}(X,\mathcal{X},\mu,T):\lambda=e(t)\text{ for some }t\in\mathbb{Q}\},\\[2mm]
\operatorname{Spec}_{\mathrm{irr}}(X,\mathcal{X},\mu,T)
&=\{\lambda\in\operatorname{Spec}(X,\mathcal{X},\mu,T):\lambda=e(t)\text{ for some }t\in\mathbb{R}\setminus\mathbb{Q}\}.
\end{aligned}
\]
We refer to eigenvalues in $\operatorname{Spec}_{\mathrm{rat}}(X,\mathcal{X},\mu,T)$ (resp.\ $\operatorname{Spec}_{\mathrm{irr}}(X,\mathcal{X},\mu,T)$) as \emph{rational eigenvalues} (resp.\ \emph{irrational eigenvalues}).
Clearly
\[
\operatorname{Spec}(X,\mathcal{X},\mu,T)
=\operatorname{Spec}_{\mathrm{rat}}(X,\mathcal{X},\mu,T)\sqcup\operatorname{Spec}_{\mathrm{irr}}(X,\mathcal{X},\mu,T),
\]
so both subsets are countable.
Writing $e(\mathbb{Q}):=\{e(t):t\in\mathbb{Q}\cap[0,1)\}$, we have
\begin{equation}\label{eq:rational-eQ}
\operatorname{Spec}_{\mathrm{rat}}(X,\mathcal{X},\mu,T)\subseteq e(\mathbb{Q})
\quad\text{and}\quad
\operatorname{Spec}_{\mathrm{irr}}(X,\mathcal{X},\mu,T)\subseteq\mathbb{T}\setminus e(\mathbb{Q}).
\end{equation}

We next define the Kronecker factor and the irrational Kronecker factor (sometimes also called the Kronecker algebra and the irrational Kronecker algebra).
To this end we need the following classical result (see, for example, \cite[Theorem 1.2]{Z76}).

\begin{prop}\label{zimmer}
Let $(X,\mathcal{X},\mu,T)$ be a measure-preserving system and let $H$ be a conjugation-invariant subalgebra of $L^{2}(X,\mathcal{X},\mu)$ consisting of bounded functions.
Then there exists a sub-$\sigma$-algebra $\mathcal{A}$ of $\mathcal{X}$ such that
\(
\overline{H}=L^{2}(X,\mathcal{A},\mu).
\)
Moreover, if $H$ is $U_{T}$-invariant, then $\mathcal{A}$ is $T$-invariant.
\end{prop}

Let $(X,\mathcal{X},\mu,T)$ be a system.
We first note a simple fact: if $f\in L^{2}(X,\mathcal{X},\mu)$ is an eigenfunction with eigenvalue $\lambda\in\operatorname{Spec}(X,\mathcal{X},\mu,T)$, then for every sufficiently large $P>0$ the truncated function
$f^{(P)}(x)=f(x) \cdot \mathbf{1}_{\{|f| \leq P\}}(x)$
is also an eigenfunction with the same eigenvalue $\lambda$. Clearly, $\lim_{P\rightarrow +\infty} \|f-f^{(P)}\|_{L^2(\mu)}=0$.

Thus, If we let \( H \) be the set of the zero function together with all bounded eigenfunctions of \( (X,\mathcal{X},\mu,T) \), then \( H \) is a conjugation-invariant subalgebra of \( L^{2}(X,\mathcal{X},\mu) \) satisfying \( U_{T}(H)=H \), and its closure \( \overline{H} \) in \( L^{2}(X,\mathcal{X},\mu) \) contains every eigenfunction of \( (X,\mathcal{X},\mu,T) \).
By Proposition~\ref{zimmer}, there exists a \( T \)-invariant sub-\( \sigma \)-algebra \( \mathcal{K}(X,\mathcal{X},\mu,T) \) of \( \mathcal{X} \) such that
\(\overline{H}=L^{2}\!\bigl(X,\mathcal{K}(X,\mathcal{X},\mu,T),\mu\bigr)\).
We call this \( \sigma \)-algebra the \emph{Kronecker factor} (or \emph{Kronecker algebra}) of \( (X,\mathcal{X},\mu,T) \).
Similarly, \( \mathcal{K}_{\mathrm{rat}}(X,\mathcal{X},\mu,T) \) denotes the \emph{rational Kronecker factor} (or \emph{rational Kronecker algebra}), generated by all bounded eigenfunctions whose eigenvalues are rational.
\begin{de} A system \( (X,\mathcal{X},\mu,T) \) is called \emph{discrete spectrum} if \( L^{2}(X,\mathcal{X},\mu)\) is spanned by the eigenfunctions
equivalently, if $ \mathcal{K}(X,\mathcal{X},\mu,T)=\mathcal{X} \mod \mu$.
\end{de}
 The following result establishes the relation between maximal pattern entropy and the Kronecker factor by using Theorem~\ref{MPE-p}~(2), which was proved in \cite[Lemma~2.2 and Theorem~2.3]{HMY04}.
\begin{thm}\label{thm-se-enytopy}
Let $(X,\mathcal{X},\mu, T)$ be a system and $\alpha\in\mathcal{P}_X^\mu$. Then
\[
h^*_\mu(T,\alpha)=H_\mu(\alpha\mid\mathcal{K}(X,\mathcal{X},\mu,T)).
\]
\end{thm}
Thus, a system $(X,\mathcal{X},\mu,T)$ has discrete spectrum if and only if $h^{*}_\mu(T)=0$; this is exactly Corollary~2.2 in \cite{HY09}.

We now show that $L^{2}(X,\mathcal{K}(X,\mathcal{X},\mu,T),\mu)$ (resp.\ $L^{2}(X,\mathcal{K}_{\mathrm{rat}}(X,\mathcal{X},\mu,T),\mu)$) decomposes into the orthogonal sum of eigenspaces with eigenvalues (resp.\ rational eigenvalues).  For each $\lambda\in\operatorname{Spec}(X,\mathcal{X},\mu,T)$ let
\[
E_{\lambda}:=\bigl\{h\in L^{2}(X,\mathcal{X},\mu):U_{T}h=\lambda h\bigr\}.
\]
Then $E_{\lambda}$ is a closed subspace of $L^{2}(X,\mathcal{X},\mu)$; since the latter is separable, so is $E_{\lambda}$, and we can pick a countable orthonormal basis $\{f_{i}^{\lambda}\}_{i\in I_{\lambda}}$ of $E_{\lambda}$.  Note that $H\subseteq\bigcup_{\lambda\in\operatorname{Spec}(X,\mathcal{X},\mu,T)}E_{\lambda}\subset\overline{H}$.  As $\operatorname{Spec}(X,\mathcal{X},\mu,T)$ is countable and $E_{\lambda}\perp E_{\lambda'}$ for $\lambda\neq\lambda'$, we obtain the orthogonal decompositions
\[
L^{2}(X,\mathcal{K}(X,\mathcal{X},\mu,T),\mu)=\overline{H}=\bigoplus_{\lambda\in\operatorname{Spec}(X,\mathcal{X},\mu,T)}E_{\lambda}.
\]
Similarly,
\[
L^{2}(X,\mathcal{K}_{\mathrm{rat}}(X,\mathcal{X},\mu,T),\mu)=\bigoplus_{\lambda\in\operatorname{Spec}_{\mathrm{rat}}(X,\mathcal{X},\mu,T)}E_{\lambda}.
\]
We summarize this discussion in the following lemma; it will be used in the proof of Proposition~\ref{c-0}.

\begin{lem}\label{dec-rat-part}
Let \((X,\mathcal{X},\mu,T)\) be a system.  Then
\begin{enumerate}
\item \(\{f_{i}^{\lambda}:\lambda\in\operatorname{Spec}(X,\mathcal{X},\mu,T),\,i\in I_{\lambda}\}\) is a countable orthonormal basis of \(L^{2}\!\bigl(X,\mathcal{K}(X,\mathcal{X},\mu,T),\mu\bigr)\).

\item
Every \(u\in L^{2}\!\bigl(X,\mathcal{K}_{\mathrm{rat}}(X,\mathcal{X},\mu,T),\mu\bigr)\) can be written uniquely as
\[
u=\sum_{\lambda\in\operatorname{Spec}_{\mathrm{rat}}(X,\mathcal{X},\mu,T)}u_{\lambda},\quad u_{\lambda}\in E_{\lambda},
\]
and
\[
\|u\|_{L^{2}(\mu)}^{2}
=\sum_{\lambda\in\operatorname{Spec}_{\mathrm{rat}}(X,\mathcal{X},\mu,T)}\|u_{\lambda}\|_{L^{2}(\mu)}^{2}.
\]
\end{enumerate}
\end{lem}

\subsection{Nilsystems}
Let \(G\) be a group. For \(g, h \in G\), we write
\(
[g, h] = ghg^{-1}h^{-1}
\)
for the commutator of \(g\) and \(h\), and we write \([A, B]\) for the subgroup generated by \(\{[a, b] : a \in A, b \in B\}\). The commutator subgroups \(G_j\), \(j \geq 1\), are defined inductively by setting
\(G_1 = G\) and $G_{j+1} = [G_j, G]$.
Let \(k \geq 1\) be an integer. We say that \(G\) is \emph{\(k\)-step nilpotent} if \(G_{k+1}\) is the trivial subgroup.

Let \(G\) be a \(k\)-step nilpotent Lie group and let \(\Gamma\) be a discrete, cocompact subgroup of \(G\).  The compact manifold \(X=G/\Gamma\) is called a \emph{\(k\)-step nilmanifold}.  The group \(G\) acts on \(X\) by left translations, written \((g,x)\mapsto gx\).  The Haar measure \(\mu\) on \(X\) is the unique probability measure invariant under this action.  Fix \(t\in G\) and let \(T\) be the transformation \(x\mapsto tx\) of \(X\); then \((X,\mathcal{X},\mu,T)\) is called a \emph{\(k\)-step nilsystem}.

\begin{de}\label{de-k-step-nil}
Let \(k\in\mathbb{N}\).  A measure-preserving system \((X,\mathcal{X},\mu,T)\) is called a \emph{\(k\)-step pro-nilsystem} (or a system of order \(k\)) if it is the measure-theoretic inverse limit of \(k\)-step nilsystems.
\end{de}

By the Host-Kra seminorms \cite{HK0} one can define factors $(Z_{d-1},\mathcal{Z}_{d-1},\mu_{d-1},T_{d-1})$ for an ergodic system $(X,\mathcal{X},\mu,T)$.

\begin{de}\label{HK-factor}
Let $(X,\mathcal{X},\mu,T)$ be an ergodic  system.
For every $d\in\mathbb{N}$ there exists a $T$-invariant sub-$\sigma$-algebra $\mathcal{Z}_{d-1}\subseteq\mathcal{X}$ such that for $f\in L^{\infty}(\mu)$,
\begin{equation}\label{eq:44}
\vert\mkern-2mu\vert\mkern-2mu\vert f\vert\mkern-2mu\vert\mkern-2mu\vert_{d}=0 \quad\Longleftrightarrow\quad \mathbb{E}_\mu(f|\mathcal{Z}_{d-1})=0,
\end{equation}
where the HK-seminorms $\vert\mkern-2mu\vert\mkern-2mu\vert \cdot \vert\mkern-2mu\vert\mkern-2mu\vert_{k}$ are defined inductively as follows: for $f \in L^{\infty}(\mu)$, we let $\vert\vert\vert f\vert\mkern-2mu\vert\mkern-2mu\vert_{1} := \left|\int f\, d\mu\right|$ and
	\[
	\vert\mkern-2mu\vert\mkern-2mu\vert f \vert\mkern-2mu\vert\mkern-2mu\vert_{k+1}^{2^{k+1}} :=\lim_{N\to\infty}\frac{1}{N} \sum_{n=1}^N \vert\vert\vert\overline{f} \cdot T^{n}f \vert\mkern-2mu\vert\mkern-2mu\vert_{k}^{2^{k}} \quad \text{for } k \in \mathbb{N},
	\]
		where all limits can be shown to exist.

The factor system associated with $\mathcal{Z}_{d-1}$ is denoted $(Z_{d-1},\mathcal{Z}_{d-1},\mu_{d-1},T_{d-1})$.
\end{de}
The sub-$\sigma$-algebra $\mathcal{Z}_{d-1}$ is called the \emph{factor of order $d-1$} of $(X,\mathcal{X},\mu,T)$.
It should be noted that $\mathcal{Z}_{1}$ is precisely the Kronecker factor $\mathcal{K}(X,\mathcal{X},\mu,T)$. In \cite{HK0} it was shown that \( (Z_{d-1}, \mathcal{Z}_{d-1}, \mu_{d-1}, T_{d-1}) \) has a very nice structure.

\begin{thm}[See {\cite[Theorem 10.1]{HK0}}]\label{cs-pro-nil}
Let \( (X, \mathcal{X}, \mu, T) \) be an ergodic measure-preserving system and let \( d \in \mathbb{N} \). Then the system \( (Z_{d-1}, \mathcal{Z}_{d-1}, \mu_{d-1}, T_{d-1}) \) is a \( (d-1) \)-step pro-nilsystem.
\end{thm}

An \emph{$\infty$-step pro-nilsystem} was introduced and studied in \cite{DDMSY}. By construction the factors $\mathcal{Z}_{k}$ ($k\ge 1$) form an increasing sequence (see \cite[Lemma 3.9]{HK0}):
\[
\mathcal{Z}_{1}\subseteq\mathcal{Z}_{2}\subseteq\dots\subseteq\mathcal{X}.
\]
Let $\mathcal{Z}_{\infty}:=\bigvee_{k\in\mathbb{N}}\mathcal{Z}_{k}$ and  let $(X_{\infty},\mathcal{X}_{\infty},\mu_{\infty},T_\infty)$ be the factor associated with $\mathcal{Z}_{\infty}$.
This system $(X_{\infty},\mathcal{X}_{\infty},\mu_{\infty},T_\infty)$ is the inverse limit of the sequence $(X_{k},\mathcal{X}_{k},\mu_{k},T_k)_{k\in\mathbb{N}}$ and is called an \emph{$\infty$-step pro-nilfactor} of $(X,\mathcal{X},\mu,T)$.

\begin{de} A system $(X,\mathcal{X},\mu,T)$ is an \emph{ergodic $\infty$-step pro-nilsystem} if it is ergodic and  $\mathcal{Z}_{\infty}=\mathcal{X} \mod \mu$.
In this case, $(X,\mathcal{X},\mu,T)$ is measure-theoretically isomorphic to the inverse limit of the sequence $(X_{k},\mathcal{X}_{k},\mu_{k},T_k)_{k\in\mathbb{N}}$.
\end{de}

\begin{rem}\label{rem-smooth-nil} We need to point out that by Theorem~\ref{cs-pro-nil} and Definition \ref{de-k-step-nil}, if $(X,\mathcal{X},\mu,T)$ is an ergodic $\infty$-step pro-nilsystem, then it is measure-theoretically isomorphic to the inverse limit of a nilsystem sequence $(X_{k},\mathcal{X}_{k},\mu_{k},T_k)_{k\in\mathbb{N}}$, where each $(X_{k},\mathcal{X}_{k},\mu_{k},T_k)$ is an ergodic $s_k$-step nilsystem for some $s_k\in\mathbb{N}$.
\end{rem}

Next, we introduce smooth \(k\)-step nilsequences, and a direct theorem is taken from \cite{HK}.
\begin{de}
Let $k\in \mathbb{N}$ and $(X,\mathcal{X},\mu,T)$ be a $k$-step nilsystem, $f:X\to\mathbb{C}$ a continuous function, $\tau\in G$, and $x_0\in X$.
The sequence $\bigl(f(\tau^n x_0)\bigr)_{n\in\mathbb{Z}}$ is called a \emph{basic $k$-step nilsequence}.
If, in addition, the function $f\in C^{\infty}(X)$, then $\bigl(f(\tau^n x_0)\bigr)_{n\in\mathbb{Z}}$ is called a \emph{smooth $k$-step nilsequence}.
\end{de}

Let $k\ge 2$ be an integer. In Section~5.3 in \cite{HK}, for every $(k-1)$-step nilmanifold $X$ Host and Kra define a norm $\vert\mkern-2mu\vert\mkern-2mu\vert \cdot \vert\mkern-2mu\vert\mkern-2mu\vert_k^*$ on the space $C^{\infty}(X)$ of smooth functions on $X$ (see Definition 5.5 and Proposition 5.6 in \cite{HK}).
Let ${\bf b}=(b_{n})_{n\in\mathbb{Z}}$ be a smooth $(k-1)$-step nilsequence. Then there exists an ergodic $(k-1)$-step nilsystem (see \cite[Corollary~3.3]{HK}), a smooth function $f$ on $X$, and a point $x_{0}\in X$ such that
\[
b_{n}=f(T^{n}x_{0}) \quad\text{for every } n\in\mathbb{Z}.
\]
The same sequence ${\bf b}$ can be represented in this way in several manners, with different systems, different starting points, and different functions, but  Corollary~5.8 in \cite{HK} shows that all associated functions $f$ have the same norm $\vert\mkern-2mu\vert\mkern-2mu\vert f\vert\mkern-2mu\vert\mkern-2mu\vert_{k}^*$. Therefore we can define (see \cite[Definition 5.9]{HK})
$\vert\mkern-2mu\vert\mkern-2mu\vert {\bf b}\vert\mkern-2mu\vert\mkern-2mu\vert_{k}^* := \vert\mkern-2mu\vert\mkern-2mu\vert f\vert\mkern-2mu\vert\mkern-2mu\vert_{k}^*$,
where $f$ is any of the possible functions.

The following Theorem comes from Theorem 2.13 and Corollary 3.10 in \cite{HK}.
\begin{thm}[Direct Theorem]\label{thm-direct} Let $(Y,\mathcal{Y},\nu,R)$ an ergodic system and $k\ge 2$. If $h\in L^\infty(Y,\mathcal{Y},\nu)$,  then for
$\nu$-a.e. $y\in Y$,
$$\limsup_{N\rightarrow +\infty}|\frac{1}{N}\sum_{n=1}^N h(R^ny)b_n|\le
\vert\mkern-2mu\vert\mkern-2mu\vert h \vert\mkern-2mu\vert\mkern-2mu\vert_k \vert\mkern-2mu\vert\mkern-2mu\vert {\bf b} \vert\mkern-2mu\vert\mkern-2mu\vert_k^* $$
for every a smooth $(k-1)$-step nilsequence ${\bf b}=(b_{n})_{n\in\mathbb{Z}}$. In particular, if $\vert\mkern-2mu\vert\mkern-2mu\vert h \vert\mkern-2mu\vert\mkern-2mu\vert_{k}=0$, then
$\nu$-a.e. $y\in Y$,
$$\lim \frac{1}{N}\sum_{n=1}^N h(R^ny)b_n=0$$
for every a smooth $(k-1)$-step nilsequence ${\bf b}=(b_{n})_{n\in\mathbb{Z}}$.
\end{thm}	
	
\section{Disjoint result}\label{sec-dis}
In this section our aim is to establish a useful disjointness result (Theorem~\ref{thm-zero-joining} below) by exploiting the structure of the Furstenberg systems of the M\"{o}bius and Liouville functions, together with a lemma of Frantzikinakis and Host \cite[Lemma~6.1]{FHost}.
To this end we recall the following notion of orthogonality.

Let $(X,T)$ be a TDS and $\mu\in\mathcal{M}(X,T)$.
For $f\in L^{2}(X,\mathcal{X},\mu)$ and a sub-$\sigma$-algebra $\mathcal{C}$ of $\mathcal{X}$, we say that \emph{$f$ is orthogonal to $\mathcal{C}$ in $L^{2}(X,\mathcal{X},\mu)$} (abbreviated as $f\perp\mathcal{C}$ in $L^{2}(X,\mathcal{X},\mu)$) if
$\mathbb{E}_{\mu}(f\mid\mathcal{C})=0 \quad \mu\text{-a.e.}$
\begin{lem}\label{lem-de-rat}
Let $(X,\mathcal{X},\mu,T)$ be a system and let $\mu=\int_{\Omega}\mu_{\omega}\,d\xi(\omega)$ be its ergodic decomposition.
If $f\in L^{\infty}(X,\mathcal{X},\mu)$ satisfies $f\perp\mathcal{K}_{\mathrm{rat}}(X,\mathcal{X},\mu,T)$ in $L^{2}(X,\mathcal{X},\mu)$, then for $\xi$-a.e.\ $\omega\in\Omega$ one has $f\in L^\infty(X,\mathcal{X},\mu_\omega)$ and $f\perp\mathcal{K}_{\mathrm{rat}}(X,\mathcal{X},\mu_{\omega},T)$ in $L^{2}(X,\mathcal{X},\mu_{\omega})$.
\end{lem}
\begin{proof} First, it is easy to see that if $f\in L^{\infty}(X,\mathcal{X},\mu)$, then for $\xi$-a.e.\ $\omega\in\Omega$ one also has $f\in L^\infty(X,\mathcal{X},\mu_\omega)$.

Fix $\nu\in\mathcal{M}(X,T)$ and $p\in\mathbb{N}$, and define
\[
\mathcal{K}_{p}(X,\mathcal{X},\nu,T)=\{A\in\mathcal{X}:T^{-p}A=A\}.
\]
It is clear that the sub-$\sigma$-algebras $\mathcal{K}_{p!}(X,\mathcal{X},\nu,T)$, $p=1,2,\dots$, form an increasing sequence and that every eigenfunction with a rational eigenvalue is measurable with respect to $\bigvee_{p\in\mathbb{N}}\mathcal{K}_{p!}(X,\mathcal{X},\nu,T)$.
Hence, by definition,
\[
\mathcal{K}_{\mathrm{rat}}(X,\mathcal{X},\nu,T)\subseteq\bigvee_{p\in\mathbb{N}}\mathcal{K}_{p!}(X,\mathcal{X},\nu,T).
\]
Conversely, since each $\mathcal{K}_{p!}(X,\mathcal{X},\nu,T)$ is contained in $\mathcal{K}_{\mathrm{rat}}(X,\mathcal{X},\nu,T)$, we obtain
\[
\mathcal{K}_{\mathrm{rat}}(X,\mathcal{X},\nu,T)=\bigvee_{p\in\mathbb{N}}\mathcal{K}_{p!}(X,\mathcal{X},\nu,T).
\]
	
Fix \( p \in \mathbb{N} \) and \( f \in L^\infty(X, \mathcal{X}, \mu) \) orthogonal to \( \mathcal{K}_p(X, \mathcal{X}, \mu, T) \) in \( L^2(X, \mathcal{X}, \mu) \).
By the Birkhoff ergodic theorem, for \( \mu \)-a.e.\ \( x \in X \),
\[
\lim_{N \to \infty} \frac{1}{N} \sum_{n=0}^{N-1} f(T^{pn}x) = \mathbb{E}_\mu(f \mid \mathcal{K}_p(X, \mathcal{X}, \mu, T))(x) = 0.
\]
Let \( B \) be the measurable set of all points \( x \in X \) for which the above limit holds; then \( B \) has full measure.
Hence, for \( \xi \)-a.e.\ \( \omega \in \Omega \), we have \( \mu_\omega(B) = 1 \).
Applying the Birkhoff ergodic theorem again, for \( \xi \)-a.e.\ \( \omega \in \Omega \) and \( \mu_\omega \)-a.e.\ \( x \in B \),
\[
\lim_{N \to \infty} \frac{1}{N} \sum_{n=0}^{N-1} f(T^{pn}x) = \mathbb{E}_{\mu_\omega}(f \mid \mathcal{K}_p(X, \mathcal{X}, \mu_\omega, T))(x).
\]
Combining this with the previous equation, for \( \xi \)-a.e.\ \( \omega \in \Omega \) and \( \mu_\omega \)-a.e.\ \( x \in B \),
\[
\mathbb{E}_{\mu_\omega}(f \mid \mathcal{K}_p(X, \mathcal{X}, \mu_\omega, T))(x) = 0.
\]
Thus, \( f \) is orthogonal to \( \mathcal{K}_p(X, \mathcal{X}, \mu_\omega, T) \) in \( L^2(X, \mathcal{X}, \mu_\omega) \) for \( \xi \)-a.e.\ \( \omega \in \Omega \).

Now let \( f \in L^\infty(X, \mathcal{X}, \mu) \) be orthogonal to \( \mathcal{K}_{\mathrm{rat}}(X, \mathcal{X}, \mu, T) \) in \( L^2(X, \mathcal{X}, \mu) \).
Then \( f \) is orthogonal to \( \mathcal{K}_{p!}(X, \mathcal{X}, \mu, T) \) in \( L^2(X, \mathcal{X}, \mu) \) for every \( p \in \mathbb{N} \).
Hence, for each \( p \) there exists a full-measure subset \( \Omega_p \subseteq \Omega \) such that \( f \) is orthogonal to \( \mathcal{K}_{p!}(X, \mathcal{X}, \mu_\omega, T) \) in \( L^2(X, \mathcal{X}, \mu_\omega) \) for all \( \omega \in \Omega_p \).
The intersection \( \bigcap_{p \in \mathbb{N}} \Omega_p \) still has full measure, and for every \( \omega \in \bigcap_{p \in \mathbb{N}} \Omega_p \) we have
\[
f \perp \mathcal{K}_{p!}(X, \mathcal{X}, \mu_\omega, T) \quad \text{in } L^2(X, \mathcal{X}, \mu_\omega) \quad \text{for all } p \in \mathbb{N}.
\]
Since \( \mathcal{K}_{\mathrm{rat}}(X, \mathcal{X}, \mu_\omega, T) = \bigvee_{p \in \mathbb{N}} \mathcal{K}_{p!}(X, \mathcal{X}, \mu_\omega, T) \), it follows that
\[
f \perp \mathcal{K}_{\mathrm{rat}}(X, \mathcal{X}, \mu_\omega, T) \quad \text{in } L^2(X, \mathcal{X}, \mu_\omega) \quad \text{for } \xi\text{-a.e.\ } \omega \in \Omega,
\]
which proves Lemma~\ref{lem-de-rat}.
	\end{proof}
The following lemma is taken from \cite[Lemma~6.1]{FHost}.
\begin{lem}\label{lemma-eig-dis}
Let $(X,\mathcal{X},\mu,T)$ be an ergodic $\infty$-step pro-nilsystem and $(Y,\mathcal{Y},\nu,R)$ an ergodic system.
\begin{itemize}
\item[(i)] If the two systems have disjoint irrational spectrum, then for every joining $\tau$ of the two systems and every function $f\in L^{\infty}(X,\mathcal{X},\mu)$ satisfying $f\perp\mathcal{K}_{\mathrm{rat}}(X,\mathcal{X},\mu,T)$ in $L^{2}(X,\mathcal{X},\mu)$, we have
\[
\int f(x)\,g(y)\,d\tau(x,y)=0
\]
for every $g\in L^{\infty}(Y,\mathcal{Y},\nu)$.

\item[(ii)] If the two systems have disjoint spectrum except possibly for $1$, i.e.,
\[
\bigl(\operatorname{Spec}(X,\mathcal{X},\mu,T)\cap\operatorname{Spec}(Y,\mathcal{Y},\nu,R)\bigr)\setminus\{1\}=\emptyset,
\]
then they are disjoint.
\end{itemize}
\end{lem}
Following the argument in \cite[Lemma 6.1]{FHost}, we make a slight modification to above result Lemma \ref{lemma-eig-dis} (i).
\begin{lem}\label{lemma-eig-dis-1}
	Let $(X,\mathcal{X},\mu,T)$ be an ergodic $\infty$-step pro-nilsystem and $(Y,\mathcal{Y},\nu,R)$ an ergodic system.
	If the two systems have disjoint irrational spectrum, then for every joining $\tau$ of the two systems and every function $g\in L^{\infty}(Y,\mathcal{Y},\nu)$ satisfying $g\perp\mathcal{K}_{\mathrm{rat}}(Y,\mathcal{Y},\nu,R)$ in $L^{2}(Y,\mathcal{Y},\nu)$, we have
		\[
		\int f(x)\,g(y)\,d\tau(x,y)=0
		\]
		for every $f\in L^{2}(X,\mathcal{X},\mu)$.
	\end{lem}
\begin{proof} Since $(X,\mathcal{X},\mu,T)$ is an ergodic $\infty$-step pro-nilsystem, by Remark \ref{rem-smooth-nil} we can identify it (up to measure-theoretic isomorphism) with the inverse limit
\[
(X,\mathcal{X},\mu,T)=\varprojlim(X_{j},\mathcal{X}_j,\mu_{j},T_j),
\]
where each $(X_{j},\mathcal{X}_j,\mu_{j},T_j)$ is an ergodic $s_j$-step nilsystem for some $s_j\in \mathbb{N}$.  Let $\pi_{j}\colon X\to X_{j}$ be the corresponding factor maps.  Then for every $j$ the image $\tau_{j}$ of $\tau$ under $\pi_{j}\times\mathrm{id}\colon X\times Y\to X_{j}\times Y$ is a joining of $(X_{j},\mathcal{X}_j,\mu_{j},T_j)$ and $(Y,\mathcal{Y},\nu,R)$.

Fixed $f\in L^{2}(X,\mathcal{X},\mu)$. Then there exists $f_j\in C^{\infty}(X_j)$, $j\in \mathbb{N}$ such that
$$\lim_{j\rightarrow +\infty} \|f-f_j\circ \pi_j\|_{L^2(\mu)}=0,$$
Thus
\begin{align*}
\lim_{j\to\infty}\int f_j(x)\,g(y)\,d\tau_{j}(x,y)&=
\lim_{j\to\infty}\int (f_j\circ \pi_j)(x)\,g(y)\,d\tau(x,y)\\
&=\int f(x)\,g(y)\,d\tau(x,y).
\end{align*}
Hence to show $\int f(x)\,g(y)\,d\tau(x,y)=0$ it is sufficient to show that
for each $j\in \mathbb{N}$,
$\int f_j(x)\,g(y)\,d\tau_{j}(x,y)=0$.

Fix \(j\in\mathbb{N}\).  We now show that \(\displaystyle\int f_j(x)\,g(y)\,d\tau_{j}(x,y)=0\). Let \(\mathcal{Z}_{s_j}(R)\) be the factor of order \(s_j\) of \((Y,\mathcal{Y},\nu,R)\) (see Definition~\ref{HK-factor}).
Set \(\tilde{g}=g-\mathbb{E}_\nu\bigl(g\bigm|\mathcal{Z}_{s_j}(R)\bigr)\); then $\mathbb{E}_\nu(g\bigm|\mathcal{Z}_{s_j}(R))\in L^{\infty}(Y,\mathcal{Y},\nu)$, \(\tilde{g}\in L^{\infty}(Y,\mathcal{Y},\nu)\) and $\mathbb{E}_\nu\bigl(\tilde{g}\bigm|\mathcal{Z}_{s_j}(R)\bigr)=0$. Hence $\vert\mkern-2mu\vert\mkern-2mu\vert \tilde{g} \vert\mkern-2mu\vert\mkern-2mu\vert_{s_j+1}=0$.

Next since \((X_j,\mathcal{X}_j,\mu_j,T_j)\) is an \(s_j\)-step nilsystem and \(f_j\in C^\infty(X_j)\), the sequence \(\bigl(f_j(T_j^n x)\bigr)_{n\in\mathbb Z}\) is a smooth $s_j$-nilsequence for every \(x\in X_j\).
Moreover, note that  $\vert\mkern-2mu\vert\mkern-2mu\vert \tilde{g}\vert\mkern-2mu\vert\mkern-2mu\vert_{s_j+1}=0$, by Theorem~\ref{thm-direct}, there exists a set \(Y_0\in  \mathcal{Y}\) with \(\nu(Y_0)=1\) such that for every \(y\in Y_0\),
\begin{equation}\label{limit0}
\lim_{N\to\infty}\frac1N\sum_{n=1}^N f_j(T_j^n x)\,\tilde g(R^n y)=0\end{equation}
for every \(x\in X_j\).
As \(\tau_j(X_j\times Y_0)=1\) and \(\tau_j\) is \((T_j\times R)\)-invariant, the dominated-convergence theorem and \eqref{limit0} imply
\[
\int f_j(x)\,\tilde g(y)\,d\tau_j(x,y)
=\lim_{N\to\infty}\frac1N\sum_{n=1}^N\int f_j(T_j^n x)\,\tilde g(R^n y)\,d\tau_j(x,y)
=0.
\]

Consequently,
\begin{equation}\label{reduce-eq-s-step-nil}
\int f_j(x)\,g(y)\,d\tau_j(x,y)
=\int f_j(x)\,\mathbb E_\nu(g\mid\mathcal Z_{s_j}(R))\,d\tau_j(x,y).
\end{equation}

Let \((Y_{s_j}, \mathcal{Y}_{s_j}, \nu_{s_j}, R_{s_j})\) be the factor system of \((Y,\mathcal{Y},\nu,R)\) associated with the sub-\(\sigma\)-algebra \(\mathcal{Z}_{s_j}(R)\).
Then \((Y_{s_j}, \mathcal{Y}_{s_j}, \nu_{s_j}, R_{s_j})\) is an \(s_j\)-step pro-nilsystem (see Theorem~\ref{cs-pro-nil}).
Let \(q_j\colon (Y,\mathcal{Y},\nu,R)\to (Y_{s_j}, \mathcal{Y}_{s_j}, \nu_{s_j}, R_{s_j})\) be the corresponding factor map; thus \(\mathcal{Z}_{s_j}(R)=q_j^{-1}(\mathcal{Y}_{s_j})\mod\nu\).

Hence we can regard \(\mathbb{E}_\nu(g\mid\mathcal{Z}_{s_j}(R))\) as a function in \(L^\infty(Y_{s_j},\mathcal{Y}_{s_j},\nu_{s_j})\); that is, there exists \(g_j\in L^\infty(Y_{s_j},\mathcal{Y}_{s_j},\nu_{s_j})\) such that \(\mathbb{E}_\nu(g\mid\mathcal{Z}_{s_j}(R))=g_j\circ q_j\).

Then the image \(\tau_j^Y\) of \(\tau_j\) under \(\mathrm{id}\times q_j\colon X_j\times Y\to X_j\times Y_{s_j}\) is a joining of \((X_j,\mathcal{X}_j,\mu_j,T_j)\) and \((Y_{s_j},\mathcal{Y}_{s_j},\nu_{s_j},R_{s_j})\).
Using \(\tau_j^Y\) and \(g_j\) we have
\[
\int f_j(x)\,\mathbb{E}_\nu(g\mid\mathcal{Z}_{s_j}(R))\,d\tau_j(x,y)
=\int f_j(x)\,g_j\circ q_j\,d\tau_j(x,y)
=\int f_j(x)\,g_j(y)\,d\tau_j^Y(x,y).
\]
Combining this with \eqref{reduce-eq-s-step-nil} gives
\begin{equation}\label{reduce-eq-s-step-nil-1}
\int f_j(x)\,g(y)\,d\tau_j(x,y)
=\int f_j(x)\,g_j(y)\,d\tau_j^Y(x,y).
\end{equation}

Since \(g\perp\mathcal{K}_{\mathrm{rat}}(Y,\mathcal{Y},\nu,R)\) in \(L^{2}(Y,\mathcal{Y},\nu)\) and \(\mathcal{K}_{\mathrm{rat}}(Y,\mathcal{Y},\nu,R)\subset \mathcal{Z}_{s_j}(R)\), it follows that
\[
\mathbb{E}_\nu\big(\mathbb{E}_\nu(g|\mathcal{Z}_{s_j}(R))|\mathcal{K}_{\mathrm{rat}}(Y,\mathcal{Y},\nu,R)\big)
=\mathbb{E}_\nu \big(g|\mathcal{K}_{\mathrm{rat}}(Y,\mathcal{Y},\nu,R)\big)=0 \text{ for }\nu\text{-a.e.}
\]
Note that \(\mathbb{E}_\nu(g|\mathcal{Z}_{s_j}(R))=g_j\circ q_j\) and
\[
\mathcal{K}_{\mathrm{rat}}(Y,\mathcal{Y},\nu,R)=q_j^{-1}(\mathcal{K}_{\mathrm{rat}}(Y_{s_j},\mathcal{Y}_{s_j},\nu_{s_j},R_{n_j})) \mod \nu.
\]
Hence
\begin{align*}
\mathbb{E}_{\nu_{s_j}}\big(g_j|\mathcal{K}_{\mathrm{rat}}(Y_{s_j},\mathcal{Y}_{s_j},\nu_{s_j},R_{n_j})\big)\circ q_j&=\mathbb{E}_\nu\big(g_j\circ q_j|q_j^{-1}(\mathcal{K}_{\mathrm{rat}}(Y_{s_j},\mathcal{Y}_{s_j},\nu_{s_j},R_{n_j}))\big)\\
&=\mathbb{E}_\nu\big(\mathbb{E}_\nu(g|\mathcal{Z}_{s_j}(R))|\mathcal{K}_{\mathrm{rat}}(Y,\mathcal{Y},\nu,R)\big)\\
&=0
\end{align*}
for \(\nu\)-a.e. Thus
\[
\mathbb{E}_{\nu_{s_j}}\big(g_j|\mathcal{K}_{\mathrm{rat}}(Y_{s_j},\mathcal{Y}_{s_j},\nu_{s_j},R_{n_j})\big)=0
\]
for \(\nu_{s_j}\)-a.e., and so \(g_j\perp \mathcal{K}_{\mathrm{rat}}(Y_{s_j},\mathcal{Y}_{s_j},\nu_{s_j},R_{s_j})\) in \(L^2(Y_{s_j},\mathcal{Y}_{s_j},\nu_{s_j})\).

Note that \((X,\mathcal{X},\mu,T)\) and \((Y,\mathcal{Y},\nu,R)\) have disjoint irrational spectrum. Then \((X_j,\mathcal{X}_j,\mu_j,T_j)\) and \((Y_{s_j},\mathcal{Y}_{s_j},\nu_{s_j},R_{s_j})\) also have disjoint irrational spectrum. Hence, by Lemma~\ref{lemma-eig-dis} applied to \(g_j,f_j,\tau_j^Y\), we have
\begin{equation}\label{key-eq-ccc}
\int f_j(x)\,g_j(y)\,d\tau_j^Y(x,y)=0.
\end{equation}
Now \(\int f_j(x)\,g(y)\,d\tau_j(x,y)=0\) follows from \eqref{reduce-eq-s-step-nil-1} and \eqref{key-eq-ccc}. This completes the proof of Lemma~\ref{lemma-eig-dis-1}.
\end{proof}

\begin{prop}\label{prop-dis}
	Let $(X,\mathcal{X},\mu,T)$ be an ergodic system isomorphic to direct product of an ergodic $\infty$-step pro-nilsystem and a Bernoulli system. Let $(Y,\mathcal{Y},\nu,R)$ be an ergodic system with  zero entropy. If the two systems have disjoint irrational spectrum, then for every joining $\tau$ of the two systems, and every function $g\in L^{\infty}(Y,\mathcal{Y},\nu)$ satisfying $g\perp K_{\mathrm{rat}}(Y,\mathcal{Y},\nu,R)$ in $L^2(Y,\mathcal{Y},\nu)$, we have
		\[
		\int f(x)\,g(y)\,d\tau(x,y) = 0
		\]
		for every $f\in L^{2}(X,\mathcal{X},\mu)$.
\end{prop}
\begin{proof} Let \(g\in L^{\infty}(Y,\mathcal{Y},\nu)\) satisfying \(g\perp \mathcal{K}_{\mathrm{rat}}(Y,\mathcal{Y},\nu,R)\) in \(L^2(Y,\mathcal{Y},\nu)\).

By assumption, \((X,\mathcal{X},\mu,T)\) is the direct product of an ergodic \(\infty\)-step nilsystem \((X',\mathcal{X}',\mu',T')\) and a Bernoulli system \((W,\mathcal{W},\eta,S)\).  After identifying \(X\) with \(W\times X'\), for a given joining \(\tau\) of \((W\times X',\mathcal{W}\times\mathcal{X}',\eta\times\mu',S\times T')\) and \((Y,\mathcal{Y},\nu,R)\), we shall show that
\begin{equation}\label{eq:41}
\int f(w,x')\,g(y)\,d\tau(w,x',y)=0
\end{equation}
for every \(f\in L^{2}(W\times X',\mathcal{W}\times\mathcal{X}',\eta\times\mu')\).

By \(L^{2}(\eta\times\mu')\)-approximation, it suffices to verify \eqref{eq:41} for \(f(w,x')=f_{1}(w)\,f_{2}(x')\) with \(f_{1}\in L^{2}(W,\mathcal{W},\eta)\) and \(f_{2}\in L^{2}(X',\mathcal{X}',\mu')\).  Let
\[
p_{W\times X'}:W\times X'\times Y\to W\times X',\text{ and }
p_{Y}:W\times X'\times Y\to Y
\]
be the corresponding coordinate projections.  Since \(\tau\) is a joining of \((W\times X',\mathcal{W}\times\mathcal{X}',\eta\times\mu',S\times T')\) and \((Y,\mathcal{Y},\nu,R)\), we have $(W\times X^{\prime}\times Y,\mathcal{W}\times \mathcal{X'}\times \mathcal{Y},\tau,S\times T'\times R)$ is system,
$(p_{W\times X'})_{*}(\tau)=\eta\times\mu'$, and  $(p_{Y})_{*}(\tau)=\nu$.

This implies that $(p_W)_*(\tau)=\eta$, $(p_{X'})_*(\tau)=\mu'$, and
\begin{equation}\label{pw-proj}
p_W:(W\times X'\times Y,\mathcal{W}\times\mathcal{X}'\times \mathcal{Y},\tau,S\times T'\times R)\to (W,\mathcal{W},\eta,S)
\end{equation}
and
$p_{X'}:(W\times X'\times Y,\mathcal{W}\times\mathcal{X}'\times \mathcal{Y},\tau,S\times T'\times R)\to(X',\mathcal{X}',\mu',T')$
are both factor maps,
where $p_{W}:W\times X'\times Y\to W$ and $p_{X'}:W\times X'\times Y\to X'$ are the corresponding coordinate projections.

Moreover, let
$$p_{X'\times Y}:W\times X'\times Y\to X'\times Y$$
be the coordinate projection and set $\rho=(p_{X'\times Y})_{*}(\tau)$; then
\begin{small}
\begin{equation}\label{p23-proj}
p_{X'\times Y}:(W\times X'\times Y,\mathcal{W}\times\mathcal{X}'\times \mathcal{Y},\tau,S\times T'\times R)\to(X'\times Y,\mathcal{X}'\times \mathcal{Y},\rho,T'\times R)
\end{equation}
\end{small}
is a factor map.

Let $q_{X'}:X'\times Y\to X'$ and $q_Y:X'\times Y\to Y$ be the corresponding coordinate projections.
Since $p_{X'}=q_{X'}\circ p_{X'\times Y}$ and $p_{Y}=q_{Y}\circ p_{X'\times Y}$, we have
$$(q_{X'})_*(\rho)=(q_{X'})_*((p_{X'\times Y})_{*}(\tau))=(p_{X'})_*(\tau)=\mu'$$
and
$$(q_Y)_*(\rho)=(q_{Y})_*((p_{X'\times Y})_{*}(\tau))=(p_{Y})_*(\tau)=\nu.$$
Thus $\rho$ is a joining of $(X',\mathcal{X'},\mu',T')$ and $(Y,\mathcal{Y},\nu,R)$.  Moreover,
$$h_{\rho}(T'\times R)\le h_{\mu'}(T')+h_{\nu}(R)=0,$$
where the last equality uses the fact that $(Y,\mathcal{Y},\nu,R)$ has zero entropy and that every ergodic $\infty$-step pro-nilsystem has zero entropy (see for example \cite[Theorem 7.14]{DDMSY}). Thus $h_{\rho}(T'\times R)=0$.

Now by \eqref{pw-proj} and  \eqref{p23-proj} one knows that $\tau$ is a joining of $(W,\mathcal{W},\eta,S)$ and $(X'\times Y,\mathcal{X}'\times \mathcal{Y},\rho,T'\times R)$.
Since Bernoulli systems are disjoint with zero entropy systems and $(X'\times Y,\mathcal{X}\times\mathcal{Y},\rho, T'\times R)$ has zero entropy, it follows that $(W,\mathcal{W},\eta,S)$ and $(X'\times Y,\mathcal{X}'\times \mathcal{Y},\rho,T'\times R)$ are disjoint, and so $\tau=\eta\times \rho$.
Thus since for any \(f_{1}\in L^{2}(W,\mathcal{W},\eta)\) and \(f_{2}\in L^{2}(X',\mathcal{X}',\mu')\), $$\int f_1(w)f_2(x')\,g(y)\,d\tau(w,x',y)=(\int f_1 d\eta)\cdot (\int f_2(x')\,g(y)\,d\rho(x^{\prime},y)),$$ hence
we reduce \eqref{eq:41} to show that  \begin{equation}\label{eq:42}
\int f_2(x')\,g(y)\,d\rho(x^{\prime},y) = 0
\end{equation}
for every \(f_{2}\in L^{2}(X',\mathcal{X}',\mu')\).

Notice that  $\rho$ is a joining of the ergodic $\infty$-step pro-nilsystem $(X',\mathcal{X}',\mu',T')$ and the ergodic system $(Y,\mathcal{Y},\nu,R)$, and  $g\in L^{\infty}(Y,\mathcal{Y},\nu)$ satisfying $g\perp K_{\mathrm{rat}}(Y,\mathcal{Y},\nu,R)$ in $L^2(Y,\mathcal{Y},\nu)$, \eqref{eq:42} is straightforward from Lemma \ref{lemma-eig-dis-1}.
\end{proof}

In this introduction, we provide the definition of a system with an almost countable spectrum. Combining this definition with \eqref{eq:rational-eQ}, we conclude that a system \((Y,\mathcal{Y},\nu,R)\) has an almost countable spectrum if and only if there exists a countable subset \(C_\nu\subset \mathbb{T}\setminus e(\mathbb{Q})\) such that, if $\nu=\int_\Omega \nu_\omega\, d\xi(\omega)$ is the ergodic decomposition of \(\nu\), then for \(\xi\)-a.e. \(\omega\in \Omega\), we have \(\text{Spec}_{irr}(Y,\mathcal{Y},\nu_\omega,R)\subset C_\nu\).

\begin{prop}\label{prop-dis-1}
	Let $(X,\mathcal{X},\mu,T)$ be a system with almost every ergodic component isomorphic to direct products of $\infty$-step pro-nilsystems and Bernoulli systems. Let $(Y,\mathcal{Y},\nu,R)$ be a system  with  zero entropy and  almost countable spectrum. If $(X,\mathcal{X},\mu,T)$ has no  irrational spectrum then for every joining $\tau$ of the two systems, and every $g\in L^{\infty}(Y,\mathcal{Y},\nu)$ satisfying $g\perp K_{\mathrm{rat}}(Y,\mathcal{Y},\nu,R)$ in $L^2(Y,\mathcal{Y},\nu)$, we have
	\[
	\int f(x)\,g(y)\,d\tau(x,y) = 0
	\]
	for every $f\in L^2(X,\mathcal{X},\mu)$.
\end{prop}
\begin{proof}
	Fix $f\in L^2(X,\mathcal{X},\mu)$. Let $\tau=\int_\Omega \tau_\omega\, d\xi(\omega)$ be the ergodic decomposition of $\tau$, and let $p_1\colon X\times Y\to X$ and $p_2\colon X\times Y\to Y$ be the coordinate projections. Then
	\[
	\mu=\int_\Omega (p_1)_*(\tau_\omega)\, d\xi(\omega) \quad\text{and}\quad \nu=\int_\Omega (p_2)_*(\tau_\omega)\, d\xi(\omega)
	\]
	are the ergodic decompositions of $\mu$ and $\nu$, respectively. Because $(Y,\mathcal{Y},\nu,R)$ has almost countable spectrum, there exists a countable subset $C_\nu\subset \mathbb{T}\setminus e(\mathbb{Q})$ such that, for $\xi$-a.e.\ $\omega\in \Omega$,
	\[
	\mathrm{Spec}_{\mathrm{irr}}\bigl(Y,\mathcal{Y},(p_2)_*(\tau_\omega),R\bigr)\subset C_\nu.
	\]
	For any $\lambda\in C_\nu$, we have $\lambda\in\mathbb{T}\setminus e(\mathbb{Q})$ and, by assumption, $\lambda$ is not an eigenvalue of $(X,\mathcal{X},\mu,T)$. Hence (see, for example, \cite[Theorem 1.1]{Ru})
	\[
	\xi\bigl(\{\omega\in\Omega: \lambda \text{ is an eigenvalue of } (X,\mathcal{X},(p_1)_*(\tau_\omega),T)\}\bigr)=0.
	\]
	Because $C_\nu$ is countable, the set
	\[
	\widetilde{\Omega}:=\{\omega\in\Omega: \mathrm{Spec}_{\mathrm{irr}}(X,\mathcal{X},(p_1)_*(\tau_\omega),T)\cap C_\nu=\emptyset\}
	\]
	has full $\xi$-measure. Thus, combining this with Lemma~\ref{lem-de-rat}, we conclude that for   $\xi$-a.e. $\omega\in \Omega$,
	\[
	\mathrm{Spec}_{\mathrm{irr}}(X,\mathcal{X},(p_1)_*(\tau_\omega),T)\cap \mathrm{Spec}_{\mathrm{irr}}(Y,\mathcal{Y},(p_2)_*(\tau_\omega),R)=\emptyset,
	\]
	and, $g\in L^\infty\bigl((p_2)_*(\tau_\omega)\bigr)$ with $g\perp K_{\mathrm{rat}}\bigl(Y,\mathcal{Y},(p_2)_*(\tau_\omega),R\bigr)$ in $L^2\bigl(Y,\mathcal{Y},(p_2)_*(\tau_\omega)\bigr)$. For $\xi$-a.e.\ $\omega\in \Omega$, $(X,\mathcal{X},(p_1)_*(\tau_\omega),T)$ is an ergodic system isomorphic to direct product of an ergodic $\infty$-step pro-nilsystem and a Bernoulli system, $f\in L^2\bigl(X,\mathcal{X},(p_1)_*(\tau_\omega)\bigr)$ and $\tau_\omega$ is a joining of $\bigl(X,\mathcal{X},(p_1)_*(\tau_\omega),T\bigr)$ and $\bigl(Y,\mathcal{Y},(p_2)_*(\tau_\omega),R\bigr)$. By Proposition~\ref{prop-dis},
	\[
	\int f(x)\,g(y)\,d\tau_\omega(x,y) = 0.
	\]
	Therefore
	\[
	\int f(x)\,g(y)\,d\tau(x,y)
	= \int\!\!\int f(x)\,g(y)\,d\tau_\omega(x,y)\, d\xi(\omega)
	= 0,
	\]
	proving Proposition~\ref{prop-dis-1}.
\end{proof}

By a system $(X,\mathcal{X},\mu,T)$, Frantzikinakis and  Host in  \cite{FHost} construct a new system on the space $X^{\mathbb{Z}}$ by averaging the prime dilates of correlations of the system on the space $X$. Since in some cases $X$ is itself a sequence space with elements denoted by $x=(x(n))_{n\in\mathbb{Z}}$, we denote elements of $X^{\mathbb{Z}}$ by $\underline{x}=(x_{n})_{n\in\mathbb{Z}}$.

\begin{de}
	Let $(X,\mathcal{X},\mu,T)$ be a system, and let $X^{\mathbb{Z}}$ be endowed with the product $\sigma$-algebra $\mathcal{X}^{\mathbb{Z}}$ (this is just Borel $\sigma$-algebra of the compact metric $X^{\mathbb{Z}}$). We write $\widetilde{\mu}$ for the Borel probability measure on $X^{\mathbb{Z}}$ characterized as follows: For every $m\in\mathbb{N}$ and all $f_{-m},\ldots,f_{m}\in L^{\infty}(X,\mathcal{X},\mu)$, we define
	\begin{equation}
	\int_{X^{\mathbb{Z}}}\prod_{j=-m}^{m}f_{j}(x_{j})\,d\widetilde{\mu}(\underline{x}):=\mathbb{E}_{p\in\mathbb{P}}\int_{X}\prod_{j=-m}^{m}T^{pj}f_{j}\,d\mu,
	\end{equation}
	where $\mathbb{E}_{p\in\mathbb{P}}$ is the limit of average along the prime numbers.
\end{de}
The measure $\widetilde{\mu}$ is invariant under the shift transformation $\sigma$  on $X^{\mathbb{Z}}$. We say that $(X^{\mathbb{Z}},\mathcal{X}^\mathbb{Z},\widetilde{\mu},\sigma)$ is the \emph{system of arithmetic progressions with prime steps} associated to the system $(X,\mathcal{X},\mu,T)$. The following theorem is taken from \cite[Theorem 3.10 and Theorem 3.11]{FHost}
\begin{thm}\label{thm--f-h}	Let $(X,\mathcal{X},\mu,T)$ be a system, and $(X^{\mathbb{Z}},\mathcal{X}^\mathbb{Z},\widetilde{\mu},\sigma)$ is the system of arithmetic progressions with prime steps associated to the system $(X,\mathcal{X},\mu,T)$. Then the followings hold.
	\begin{itemize}
		\item[(1)] Almost every ergodic component of the system $(X^{\mathbb{Z}}, \mathcal{X}^\mathbb{Z},\widetilde{\mu}, \sigma)$, of arithmetic progressions with prime steps, is isomorphic to a direct product of an infinite-step nilsystem and a Bernoulli system;
		\item[(2)] $(X^{\mathbb{Z}}, \mathcal{X}^\mathbb{Z},\widetilde{\mu}, \sigma)$ has no irrational spectrum.
	\end{itemize}
\end{thm}
Let $\sigma$ be the left shift on $X=\{-1,0,1\}^{\mathbb{Z}}$. For a sequence $\bm{a}=\{ a_n\}_{n\in\mathbb{Z}}\in X$,  a \emph{Furstenberg system} of $a$ is a system $(X, \mathcal{X}, \mu,\sigma)$ where $\mu$ is an accumulation point of $\{\frac{1}{\log N}\sum_{j=1}^N\frac{1}{j}\delta_{\sigma^j \bm{a}}: N\in\mathbb{N}, N\ge 2 \}$. By regarding the M\"obius function $\bm{\mu}$ or the Liouville function $\bm{\lambda}$ a sequence in $\{-1,0,1\}^{\mathbb{Z}}$, the following result is proved in \cite[Proposition 3.9]{FHost}.

\begin{prop}\label{prop-1}
	A Furstenberg system $(X,\mathcal{X},\mu,T)$ of the M\"obius or the Liouville function is a factor of the associated system $(X^{\mathbb{Z}},\mathcal{X}^\mathbb{Z},\widetilde{\mu},\sigma)$ of arithmetic progressions with prime steps.
\end{prop}
\begin{rem}\label{rem-pi-continuous}Recall that $X=\{-1,0,1\}^{\mathbb{Z}}$. The factor map $\pi\colon X^\mathbb{Z} \to X$ in Proposition \ref{prop-1} is defined as follows:
 For $\underline{x} = (x_n)_{n\in\mathbb{Z}} \in X^\mathbb{Z}$, let
\[
(\pi(\underline{x}))(n) = -x_n(0).
\]
Thus, the factor map $\pi$ is continuous. For more details, one can see the proof of \cite[Proposition 3.9]{FHost}.
\end{rem}
\begin{thm}\label{thm-zero-joining}Let $(X, \mathcal{X}, \mu,\sigma)$ be a  Furstenberg system of the M\"obius or the Liouville function. Let $(Y,\mathcal{Y},\nu,R)$ be a system  with  zero entropy and  almost countable spectrum. For every joining $\tau$ of the two systems, and every $g\in L^{\infty}(Y,\mathcal{Y},\nu)$ satisfying $g\perp K_{\mathrm{rat}}(Y,\mathcal{Y},\nu,R)$ in $L^2(Y,\mathcal{Y},\nu)$, we have
	\[
	\int f(x)\,g(y)\,d\tau(x,y) = 0
	\]
	for every $f\in L^{2}(X,\mathcal{X},\mu)$.
\end{thm}
\begin{proof}By Proposition \ref{prop-1} and Remark \ref{rem-pi-continuous}, there exists a continuous factor map  $\pi$  between $(X^{\mathbb{Z}},\mathcal{X}^\mathbb{Z},\widetilde{\mu},\sigma)$  and $(X,\mathcal{X}, \mu,\sigma)$. For $f\in L^{2}(X,\mathcal{X},\mu)$, put $\widetilde{f}=f\circ\pi$ and find a joining $\widetilde{\tau}$ of $(X^{\mathbb{Z}},\mathcal{X}^\mathbb{Z},\widetilde{\mu},\sigma)$ and  $(Y,\mathcal{Y},\nu,R)$  such that $\tau=(\pi\times\text{id})_*(\widetilde{\tau})$. Then $\widetilde{f}\in L^{2}(X^{\mathbb{Z}},\mathcal{X}^\mathbb{Z},\widetilde{\mu})$. Straightforward from Proposition \ref{prop-dis-1} and Theorem \ref{thm--f-h}
	$$	\int \widetilde{f}(\underline{x})\,g(y)\,d\widetilde{\tau}(\underline{x},y) = 0.$$
This implies
	$\int f(x)\,g(y)\,d\tau(x,y) = 0,$ which
proves Theorem \ref{thm-zero-joining}.	\end{proof}

\section{Proof of Theorem \ref{thm-A}}\label{sec-proof-thm-a}
In this section our goal is to prove Theorem~\ref{thm-A}. In fact, Theorem~\ref{thm-A} is an immediate corollary of the following proposition.

\begin{prop}\label{c-0}
Let $(Y,R)$ be a TDS.
Assume there exist a point $z\in Y$, a sequence $2\le N_{1}<N_{2}<N_{3}<\cdots$ of natural numbers, and $\nu\in \mathcal{M}(Y,R)$ such that
\[
\lim_{i\to\infty}\frac{1}{\log N_i}\sum_{n=1}^{N_i}\frac{\delta_{R^n z}}{n}=\nu,
\]
where $\nu$ has zero entropy and almost countable spectrum.
Then for every $f\in C(Y)$,
\begin{equation}\label{lim-f-mobious}
\lim_{i\to\infty}\frac{1}{\log N_i}\sum_{n=1}^{N_i}\frac{f(R^{n}z)\bm{\mu}(n)}{n}=0.
\end{equation}
\end{prop}

We now proceed to give the proof of Proposition~\ref{c-0}. Let \( f \in C(Y) \).  To verify \eqref{lim-f-mobious} it suffices to show that for every \( \varepsilon > 0 \),
\begin{equation}\label{lim-f-7epsilon}
\limsup_{i \to \infty} \biggl| \frac{1}{\log N_i} \sum_{n=1}^{N_i} \frac{f(R^n z) \bm{\mu}(n)}{n} \biggr| \leq 8\varepsilon.
\end{equation}

Fix \( \varepsilon > 0 \).  Write
\(
\mathcal{T} := \{ t \in [0,1) : e(t) \in \operatorname{Spec}_{\mathrm{rat}}(Y, \mathcal{Y}, \nu, R) \}\),
which is a countable subset of \( [0,1) \cap \mathbb{Q} \).  Put
\(f_* := f - \mathbb{E}_\nu \bigl( f \big| \mathcal{K}_{\mathrm{rat}}(Y, \mathcal{Y}, \nu, R) \bigr)\).
Then \( f_* \in L^\infty(Y,\mathcal{Y},\nu) \) and \( f_* \perp \mathcal{K}_{\mathrm{rat}}(Y, \mathcal{Y}, \nu, R) \) in \( L^2(Y,\mathcal{Y},\nu) \).

By the orthogonal decomposition (see Lemma~\ref{dec-rat-part}~(2)),
\[
\mathbb{E}_\nu \bigl( f \big| \mathcal{K}_{\mathrm{rat}}(Y, \mathcal{Y}, \nu, R) \bigr) = \sum_{t \in \mathcal{T}} f_t,
\]
where each \( f_t \in E_{e(t)}:=\{h\in L^2(Y,\mathcal{K}_{\mathrm{rat}}(Y, \mathcal{Y}, \nu, R),\nu): U_Rh=e(t)h\} \) for $t\in \mathcal{T}$ and
\[
\sum_{t \in \mathcal{T}} \|f_t\|_{L^2(\nu)}^2 = \bigl\| \mathbb{E}_\nu \bigl( f \big| \mathcal{K}_{\mathrm{rat}}(Y, \mathcal{Y}, \nu, R) \bigr) \bigr\|_{L^2(\nu)}^2 < \infty.
\]
Thus \(f = f_* + \sum_{t \in \mathcal{T}} f_t\).

For \( P > 0 \) and \( t \in \mathcal{T} \) define
\[
f_t^{(P)}(x) = f_t(x) \cdot \mathbf{1}_{\{|f_t| \leq P\}}(x).
\]
Then the bounded function \( f_t^{(P)} \in E_{e(t)} \). Choose a finite set \( C \subset \mathcal{T} \) and a large \( P > 0 \) such that
\begin{equation}\label{eq-de-cont}
\biggl\| \sum_{t \in \mathcal{T}} f_t - \sum_{t \in C} f_t^{(P)} \biggr\|_{L^1(\nu)}\le \biggl\| \sum_{t \in \mathcal{T}} f_t - \sum_{t \in C} f_t^{(P)} \biggr\|_{L^2(\nu)} < \varepsilon.
\end{equation}

Since $C$ is finite, we can find a Borel measurable subset $Y_0\subset Y$ with $\nu(Y_0)=1$ such that
\begin{equation}\label{ft-rot-eq}
f_t^{(P)}(R^ny)=e(nt)f_t^{(P)}(y) \text{ for every }y\in Y_0,\, n\in \mathbb{N}, \, t\in C.
\end{equation}
By \cite[Theorem 1.3]{MRT}, there exists $L\in \mathbb{N}$ such that
\begin{equation}\label{L-cont-bound}
\limsup_{N\rightarrow +\infty} \frac{1}{\log N} \sup_{\theta\in [0,1]}\Big(\sum_{n=1}^N\frac{1}{n} \big|\frac{1}{L}\sum_{\ell=1}^{L}e(\ell \theta)\bm{\mu}(n+\ell)\big|\Big)<\frac{\epsilon}{(\#C+1)P}.
\end{equation}

By Lusin's theorem, there exists a compact subset $K\subset Y_0$ such that
\begin{align}\label{eq-002}
\nu(K)> 1-\frac{\epsilon}{(L+1)(\#C+1)P}
\end{align}
and $f_t|_K$ is continuous for each $t\in C$. We can then find a continuous function $g_t$ on $Y$ such that $g_t|_K=f_t^{(P)}|_K$ and $\|g_t\|_{\infty}:=\max_{y\in Y} |g_t(y)|\le \|f_t^{(P)}\|_{L^\infty(\nu)}\le P$.
Thus,
\begin{align}\label{eq-minius-f-g}\begin{split}
&\hskip0.5cm \|f-\sum_{t\in C}g_t-f_{*}\|_{L^1(\nu)}\\
&\le \|f-\sum_{t\in C}f_t^{(P)}-f_{*}\|_{L^1(\nu)}+\sum_{t\in C}\|f_t^{(P)}-g_t\|_{L^1(\nu)}\\
&\le\biggl\|\sum_{t\in\mathcal{T}}f_t-\sum_{t\in C}f_t^{(P)}\biggr\|_{L^1(\nu)} +\sum_{t\in C}2\|f_t^{(P)}\|_{L^\infty(\nu)}(1-\nu(K))\\
&\le 3\epsilon,
\end{split}\end{align}
where we have used \eqref{eq-de-cont} and \eqref{eq-002} in the last inequality.

We claim that
\begin{align}\label{eq-irr-small-1}\begin{split}\limsup_{i\to\infty}\biggl|\frac{1}{\log N_i}\sum_{n=1}^{N_i}\frac{ (f-\sum_{t\in C}g_t)(R^{n}z)\bm{\mu}(n)}{n}\biggr|\le 3\epsilon.
\end{split}\end{align}
To prove the claim, we assume the contrary and fix a subsequence $\{\widetilde{N}_i \}_{i\in\mathbb{N}}$ of  $\{N_i \}_{i\in\mathbb{N}}$ such that
$$\lim_{i\to\infty}\frac{1}{\log\widetilde{N}_i}\sum_{n=1}^{\widetilde{N}_i}\frac{\delta_{(\sigma^n\bm{\mu},R^nz)}}{n}=\tau \text{ for some }\tau\in \mathcal{M}(X\times Y,\sigma\times R)$$
and
\begin{align}\label{eq-irr-small-2}
\begin{split}
\limsup_{i\to\infty}\biggl|\frac{1}{\log \widetilde{N}_i}\sum_{n=1}^{\widetilde{N}_i}\frac{ (f-\sum_{t\in C}g_t)(R^{n}z)\bm{\mu}(n)}{n}\biggr|>3\epsilon,
\end{split}
\end{align}
where $X=\{ -1,0,1\}^{\mathbb{N}}$ and $\sigma:X\rightarrow X$ is the left shift.

Let $p_1:X\times Y\rightarrow X$ and $p_2:X\times Y\rightarrow Y$ be the coordinate projections. Then
$$(p_1)_*(\tau)=\lim_{i\to\infty}\frac{1}{\log\widetilde{N}_i}\sum_{n=1}^{\widetilde{N}_i}\frac{\delta_{\sigma^n\bm{\mu}}}{n}\in \mathcal{M}(X,\sigma)$$
and
$$(p_2)_*(\tau)=\lim_{i\to\infty}\frac{1}{\log\widetilde{N}_i}\sum_{n=1}^{\widetilde{N}_i}\frac{\delta_{R^nz}}{n}=\nu.$$
Thus $(X,\mathcal{X}, (p_1)_*(\tau),\sigma)$ is a Furstenberg system of the M\"{o}bius function $\bm{\mu}$, and $\tau$ is a joining of $(X,\mathcal{X},(p_1)_*(\tau),\sigma)$ and $(Y,\mathcal{Y},\nu,R)$.

Put $F(x)=x(0)$ for $x=\{x(i)\}_{i\in\mathbb{Z}}\in\{-1,0,1\}^{\mathbb{Z}}$. Then \(F\in C(X)\) and \(|F(x)|\le 1\) for every \(x\in X\). Since \((Y,\mathcal{Y},\nu,R)\) has almost countable spectrum and \(f_*\in L^\infty(Y,\mathcal{Y},\nu)\) satisfies \(f_*\perp\mathcal{K}_{\mathrm{rat}}(Y,\mathcal{Y},\nu,R)\) in \(L^2(Y,\mathcal{Y},\nu)\), Theorem~\ref{thm-zero-joining} gives
\begin{align}\label{eq-012}	
\int F(x)\,f_{*}(y)\,d\tau(x,y)=0.
\end{align}
Moreover,
\begin{align*}
&\hskip0.5cm \biggl|\int F(x)(f-\sum_{t\in C}g_t)(y)\,d\tau(x,y)\biggr|\\
&\le \biggl|\int F(x)(f-\sum_{t\in C}g_t-f_{*})(y)\,d\tau(x,y)\biggr|+\biggl|\int F(x)f_{*}(y)\,d\tau(x,y)\biggr|\\
&\le \|f-\sum_{t\in C}g_t-f_{*}\|_{L^1(\nu)}+\biggl|\int F(x)f_{*}(y)\,d\tau(x,y)\biggr|.
\end{align*}
Together with \eqref{eq-minius-f-g} and \eqref{eq-012}, this yields
\[
\biggl|\int F(x)(f-\sum_{t\in C}g_t)(y)\,d\tau(x,y)\biggr|\le 3\epsilon.
\]
Hence
\begin{align*}
&\hskip0.5cm \limsup_{i\to\infty}\biggl|\frac{1}{\log \widetilde{N}_i}\sum_{n=1}^{\widetilde{N}_i} (f-\sum_{t\in C}g_t)(R^{n}z)\bm{\mu}(n)\biggr|\\
&=\limsup_{i\to\infty}\left| \int F(x)(f-\sum_{t\in C}g_t)(y)\, d \biggl(\frac{1}{\log \widetilde{N}_i}\sum_{n=1}^{\widetilde{N}_i}\delta_{(\sigma^n\bm{\mu},R^{n}z)}\biggr)(x,y)\right|\\
&=\biggl|\int F(x)(f-\sum_{t\in C}g_t)(y)\,d\tau(x,y)\biggr|\\
&\le 3\epsilon.
\end{align*}
This contradicts \eqref{eq-irr-small-2}, so \eqref{eq-irr-small-1} holds, which proves the claim.

Now we prove the second claim. For every \(t \in C\) we assert that
\begin{align}\label{eq-rat-small-1}
\limsup_{i\to\infty}\biggl|\frac{1}{\log N_i}\sum_{n=1}^{N_i}\frac{g_t(R^{n}z)\bm{\mu}(n)}{n}\biggr|\le\frac{5\epsilon}{\#C+1}.
\end{align}
Fix \(t\in C\). The continuity of \(g_t\) guarantees an \(\epsilon_L>0\) such that
\begin{align}\label{5-2-2}
|g_t(R^iy)-g_t(R^iy')|<\frac{\epsilon}{\#C+1}
\end{align}
for \(i=0,1,\dots,L\) and any \(y,y'\in Y\) with \(d(y,y')<\epsilon_L\).

Set $K_L=\bigcap_{j=0}^{L}R^{-j}K$. Since $\nu(R^{-j}K)=\nu(K)\overset{\eqref{eq-002}}>1-\frac{\epsilon}{(L+1)(\#C+1)P}$ for $j=0,1,\cdots,L$,
one has $\nu(K_L)>1-\frac{\epsilon}{(\#C+1)P}$. Put
\[
U_L=\{y\in Y:d(y,K_L)<\epsilon_L\}\quad\text{and}\quad E_L=\{n\in\mathbb N:R^{n}z\in U_L\}.
\]
Clearly \(U_L\) is an open subset of \(Y\) containing \(K_L\). Let \(1_{U_L}\) denote the characteristic function of the open set \(U_L\). Then
\begin{align}\label{eq-5-2-4}
\begin{split}
\liminf_{i\to\infty}\frac{1}{\log N_i}\sum_{n\in E_L\cap[1,N_i]}\frac{1}{n}
&=\liminf_{i\to\infty}\frac{1}{\log N_i}\sum_{n=1}^{N_i}\frac{1_{U_L}(R^nz)}{n}\ge\nu(U_L)\\[2mm]
&\ge\nu(K_L)>1-\frac{\epsilon}{(\#C+1)P}.
\end{split}
\end{align}

For every \(n\in E_L\) pick \(z_n\in K_L\) with \(d(R^nz,z_n)<\epsilon_L\); for \(n\notin E_L\) set \(z_n=z\). For any \(i\in\mathbb N\) write briefly
\[
I_i(t,L):=\frac{1}{\log N_i}\sum_{n=1}^{N_i}\frac{1}{L}\sum_{l=1}^{L}\frac{g_t(R^{l}z_n)\bm{\mu}(n+l)}{n};
\]
then
\begin{equation}\label{main-equation}
\biggl|\frac{1}{\log N_i}\sum_{n=1}^{N_i}\frac{g_t(R^{n}z)\bm{\mu}(n)}{n}\biggr|\le
\biggl|\frac{1}{\log N_i}\sum_{n=1}^{N_i}\frac{g_t(R^{n}z)\bm{\mu}(n)}{n}-I_i(t,L)\biggr|
+|I_i(t,L)|.
\end{equation}
Next we turn to proving two lemmas.

	\begin{lem}\label{c-1}
		One can estimate the first part of inequality (\ref{main-equation}) as follows:
		$$\limsup_{i\to\infty}|\frac{1}{\log N_i}\sum_{n=1}^{N_i}\frac{g_t(R^{n}z)\bm{\mu}(n)}{n}-I_i(t,L)|\leq \frac{3 \epsilon}{\#C+1}.$$
	\end{lem}
	\begin{proof} For brevity, write $\hat{I}_i(t,L):=\frac{1}{\log N_i}\sum_{n=1}^{N_i}\frac{1}{L}\sum_{l=1}^{L}\frac{g_t(R^{n+l}z)\bm{\mu}(n+l)}{n}$.
		 Then for $i\in\mathbb{N}$ large enough, we have
		\begin{align*}
		&\hskip0.5cm |\frac{1}{\log N_i}\sum_{n=1}^{N_i}\frac{g_t(R^nz)\bm{\mu}(n)}{n}-\hat{I}_i(t,L)|\\
&\leq|\frac{1}{\log N_i}\sum_{n=1}^{N_i}\frac{g_t(R^nz)\bm{\mu}(n)}{n}-\frac{1}{\log N_i}\sum_{n=1}^{N_i}\frac{1}{L}\sum_{l=1}^{L}\frac{g_t(R^{n+l}z)\bm{\mu}(n+l)}{n+l}|\\
		&+\frac{1}{\log N_i}\sum_{n=1}^{N_i}\frac{1}{L}\sum_{l=1}^{L}|\frac{g_t(R^{n+l}z)\bm{\mu}(n+l)}{n}-\frac{g_t(R^{n+l}z)\bm{\mu}(n+l)}{n+l}|\\
		&\leq \frac{2L \|g_t\|_{\infty}}{\log N_i}+\frac{\|g_t\|_{\infty}}{\log N_i}\sum_{n=1}^{\infty}\frac{L}{n^2}.
		\end{align*}
		Note that $\sum_{n=1}^{\infty}\frac{1}{n^2}<\infty$ and $\|g_t\|_{\infty}\le P$, we have
		\begin{align}\label{5-2-7}
		\lim_{i\to\infty}|\frac{1}{\log N_i}\sum_{n=1}^{N_i}\frac{g_t(R^nz)\bm{\mu}(n)}{n}-\hat{I}_i(t,L)|=0.
		\end{align}
		
		Note that $\lim_{i\to\infty}\frac{\sum_{n\in [1,N_i]}\frac{1}{n}}{\log N_i}=1$. Thus, by $\eqref{eq-5-2-4}$, when $i\in\mathbb{N}$ large enough, we have
		\begin{align}\label{5-2-8}
		\frac{1}{\log N_i}\sum_{n\in[1,N_i]\setminus E_L}\frac{1}{n}<\frac{\epsilon}{(\#C+1)P}.
		\end{align}
		So, we have
\begin{small}		
\begin{align*}
		&\hskip0.5cm |I_i(t,L)-\hat{I}_i(t,L)|\leq\frac{1}{\log N_i}\sum_{n=1}^{N_i}\frac{1}{L}\sum_{l=1}^{L}\frac{\vert g_t(R^{l}z_n)-g_t(R^{l}(R^nz))\vert}{n}\\
		&\leq\frac{1}{\log N_i}\sum_{n\in[1,N_i]\setminus E_L}\frac{1}{L}\sum_{l=1}^{L}\frac{\vert g_t(R^{l}z_n)-g_t(R^{l}(R^nz))\vert}{n}\\
		&\ \ \ \ \   \   \   \   \   \   \    \    \    \    \      \  +\frac{1}{\log N_i}\sum_{n\in E_L\cap[1,N_i]}\frac{1}{L}\sum_{l=1}^{L}\frac{\vert g_t(R^{l}z_n)-g_t(R^{l}(R^nz))\vert}{n}\\
		&\overset{\text(\ref{5-2-2})}\leq \left(\frac{2\|g_t\|_{\infty}}{\log N_i}\sum_{n\in[1,N_i]\setminus E_L}\frac{1}{n}\right)+\left(\frac{1}{\log N_i}\sum_{n\in E_L\cap[1,N_i]}\frac{1}{n}\cdot \frac{\epsilon}{\#C+1}\right).
		\end{align*}
\end{small}
		Combining this with \eqref{5-2-8} and the fact $\|g_t\|_{\infty}\le P$, one has that for $i\in\mathbb{N}$ large enough
		\begin{equation}\label{5-2-9}
		|I_i(t,L)-\hat{I}_i(t,L)|<\frac{3\epsilon}{\#C+1}.
		\end{equation}
		Then, by$(\ref{5-2-7})$ and \eqref{5-2-9}, we have
		$\limsup\limits_{i\to\infty}|\frac{1}{\log N_i}\sum_{n=1}^{N_i}\frac{g_t(R^nz)\bm{\mu}(n)}{n}-I_i(t,L)|\leq \frac{3\epsilon}{\#C+1}$. This completes the proof of Lemma \ref{c-1}.
	\end{proof}
	\begin{lem}\label{c-2}
	One can estimate the second part of inequality (\ref{main-equation}) as follows:
		$$\limsup_{i\to\infty}|I_i(t,L)|\leq \frac{2\epsilon}{\#C+1}.$$
	\end{lem}
	\begin{proof} For each $n\in E_L$, one has $z_n\in K_L$ and so $R^j (z_n)\in K\subset Y_0$ for $j=0,1,\cdots,L$.
  Thus by \eqref{ft-rot-eq} and the fact $g_t|_K=f_t^{(P)}|_K$, one has
   $$g_t(R^jz_n)=f_t^{(P)}(R^jz_n)=e(jt)f_t^{(P)}(z_n)=e(jt)g_t(z_n) \text{ for }j=0,1,\cdots,L.$$
  Hence
  \begin{align}\label{eq-theta-con-1}
  \begin{split}
  |\frac{1}{L}\sum_{l=1}^{L}g_t(R^{l}z_n)\bm{\mu}(n+l)|&=|g_t(z_n)||\frac{1}{L}\sum_{\ell=1}^{L}e(\ell t)\bm{\mu}(n+\ell)|\\
  &\le P|\frac{1}{L}\sum_{\ell=1}^{L}e(\ell t)\bm{\mu}(n+\ell)|
  \end{split}
  \end{align}
  for $n\in E_L$.

	Note that
		\begin{align*}
		|I_i(t,L)|&\leq\frac{1}{\log N_i}\sum_{n\in  [1,N_i]\setminus E_L}\frac{1}{n}|\frac{1}{L}\sum_{l=1}^{L}g_t(R^{l}z_n)\bm{\mu}(n+l)|\\
		&\ \ \ \ \ \ +\frac{1}{\log N_i}\sum_{n\in E_L\cap [1,N_i]}\frac{1}{n}|\frac{1}{L}\sum_{l=1}^{L}g_t(R^{l}z_n)\bm{\mu}(n+l)|.
        \end{align*}
        By \eqref{eq-theta-con-1} and \eqref{5-2-8}, when $i\in \mathbb{N}$ large enough, one has
        \begin{align*}
		|I_i(t,L)|&\overset{\eqref{eq-theta-con-1}}\leq \frac{1}{\log N_i}\sum_{n\in  [1,N_i]\setminus E_L}\frac{\|g_t\|_{\infty}}{n}+\frac{P}{\log N_i}\sum_{n\in E_L\cap [1,N_i]}\frac{1}{n} |\frac{1}{L}\sum_{\ell=1}^{L}e(\ell t)\bm{\mu}(n+\ell)|\\
		&\overset{\eqref{5-2-8}}\leq \frac{\epsilon}{\#C+1}+\frac{P}{\log N_i}\sum_{n=1}^{N_i}\frac{1}{n}|\frac{1}{L}\sum_{\ell=1}^{L}e(\ell t))\bm{\mu}(n+\ell)|.
		\end{align*}
		Combining this with \eqref{L-cont-bound}, one has $\limsup_{i\to\infty}|I_i(t,L)|\leq \frac{2\epsilon}{\#C+1}$.
This finishes the proof of Lemma \ref{c-2}.
	\end{proof}
	Now, by Lemma \ref{c-1}, Lemma \ref{c-2} and (\ref{main-equation}), we have the second claim, that is \eqref{eq-rat-small-1} holds.
		Together with \eqref{eq-rat-small-1} and \eqref{eq-irr-small-1}, we obtain \eqref{lim-f-7epsilon}. Letting $\epsilon\to 0$ yields
$\displaystyle\lim_{i\to\infty}\frac{1}{\log N_i}\sum_{n=1}^{N_i}\frac{f(R^nz)\bm{\mu}(n)}{n}=0$,
which completes the proof of Proposition \ref{c-0}.

\section{Examples of Almost Countable Spectrum Systems}\label{sec-example}
In this section we exhibit several classes of TDSs whose spectrum is almost countable.
At the same time, using Theorem~\ref{thm-A}, we provide proofs of Theorems~\ref{LSC-group-extension}, \ref{LSC-susp} and \ref{LSC-bounded-m-MPE}.

\subsection{Group extension and suspension flow} First we prove the following Proposition \ref{SZ-group-extension}: a group extension of a TDS having zero entropy and only countably many ergodic measures has zero entropy and almost countable spectrum. Thus Theorem \ref{LSC-group-extension} is a direct corollary of Proposition \ref{SZ-group-extension} and Theorem~\ref{thm-A}.

\begin{prop}\label{SZ-group-extension} Let \(\pi\colon(X,T)\to(Y,R)\) be a group extension between two TDSs. If \((Y,R)\) has zero entropy and only countably many ergodic measures, then $(X,T)$ has zero entropy and almost countable spectrum.
\end{prop}
\begin{proof} First we can define a lift map \(\phi\colon\mathcal{M}^e(Y,R)\to\mathcal{M}^e(X,T)\) such that \(\pi_*(\phi(\nu))=\nu\) for every \(\nu\in\mathcal{M}^e(Y,R)\).
Since \(\pi\) is a group extension, there exists a compact subgroup \(K\) of \(\mathrm{Aut}(X,T)\) such that \(R_\pi=\{(x,kx):x\in X,\,k\in K\}\).

Next we prove
\begin{equation}\label{eq-x-y}
\mathcal{M}^e(X,T)=\{g_*\phi(\nu):\nu\in\mathcal{M}^e(Y,R),\,g\in K\}.
\end{equation}
Indeed, for any \(\nu\in\mathcal{M}^e(Y,R)\) and \(g\in K\) it is clear that \(g_*\phi(\nu)\in\mathcal{M}^e(X,T)\).
Conversely, given \(\mu\in\mathcal{M}^e(X,T)\), set \(\nu_0=\pi_*(\mu)\). Then \(\pi_*(\phi(\nu_0))=\nu_0=\pi_*(\mu)\), so
\[
\nu_0\bigl(\pi(\mathrm{Gen}(\mu))\bigr)=\nu_0\bigl(\pi(\mathrm{Gen}(\phi(\nu_0)))\bigr)=1.
\]
Hence there exist \(x\in\mathrm{Gen}(\mu)\) and \(x'\in\mathrm{Gen}(\phi(\nu_0))\) with \(\pi(x)=\pi(x')\). Choose \(g\in K\) such that \(x=gx'\). Then
\begin{align*}
\mu &=\lim_{N\to+\infty}\frac1N\sum_{n=0}^{N-1}\delta_{T^n x}
     =\lim_{N\to+\infty}\frac1N\sum_{n=0}^{N-1}\delta_{T^n gx'}\\
    &=\lim_{N\to+\infty}g_*\biggl(\frac1N\sum_{n=0}^{N-1}\delta_{T^n x'}\biggr)
     =g_*\phi(\nu_0)\\
    &\in \{g_*\phi(\nu):\nu\in\mathcal{M}^e(Y,R),\,g\in K\}.
\end{align*}
Thus \eqref{eq-x-y} holds.

Observe that
\begin{equation}\label{eq-x-y-c}
\mathrm{Spec}_{\mathrm{irr}}(X,\mathcal{X},g_*\phi(\nu),T)
=\mathrm{Spec}_{\mathrm{irr}}(X,\mathcal{X},\phi(\nu),T)
\end{equation}
is a countable set for every \(\nu\in\mathcal{M}^e(Y,R)\) and every \(g\in K\). Since \(\mathcal{M}^e(Y,R)\) is countable, the set
\[
C:=\bigcup_{\nu\in\mathcal{M}^e(Y,R)}\mathrm{Spec}_{\mathrm{irr}}(X,\mathcal{X},\phi(\nu),T)
\]
is likewise countable.

By \eqref{eq-x-y} and \eqref{eq-x-y-c} we therefore have
\begin{equation}\label{eq-0}
\mathrm{Spec}_{\mathrm{irr}}(X,\mathcal{X},\mu,T)\subset C
\quad\text{for every }\mu\in\mathcal{M}^e(X,T).
\end{equation}
Consequently \((X,T)\) has almost countable spectrum. Finally, since group extensions preserve entropy and \((Y,R)\) has zero entropy, \((X,T)\) also has zero entropy. This completing the proof of Proposition \ref{SZ-group-extension}.
\end{proof}
\begin{rem} Let $(X,T)$ be a  be a TDS with zero entropy and countable many ergodic measures. For a continuous function $h:X\to \mathbb{R}$, define $T_h$ on $X\times \mathbb{T}$ by
$$T_h(x,y)=(Tx,h(x)+y).$$
Then by Theorem \ref{LSC-group-extension} the logarithmic Sarnak conjecture holds for $(X\times\mathbb{T},T_h)$.
\end{rem}

Next we prove the following Proposition~\ref{SZ-susp}: time-one maps of continuous suspension flows of a TDS with zero entropy and only countably many ergodic measures have zero entropy and almost countable spectrum. Thus Theorem~\ref{LSC-susp} is a direct corollary of Proposition~\ref{SZ-susp} and Theorem~\ref{thm-A}.

\begin{prop}\label{SZ-susp}
Let $(X,T)$ be a TDS with zero entropy and countably many ergodic measures. For any continuous function $r:X\to(0,+\infty)$, the time-one map $(X_r,\varphi_1)$ has zero entropy and almost countable spectrum.
\end{prop}
\begin{proof}
Denote by $\mathcal{M}(X_r,\Phi)$ the space of $\Phi$-invariant probability measures on $X_r$. It is a classical result of Ambrose--Kakutani \cite{AK} that, writing $\mathrm{Leb}$ for one-dimensional Lebesgue measure, the map $\mu\mapsto\tilde{\mu}$ defined by
\(
\tilde{\mu}:=(\mu\times\mathrm{Leb})|_{X_r}
\)
is a bijection from $\mathcal{M}(X,T)$ onto $\mathcal{M}(X_r,\Phi)$; explicitly, for every $f\in C(X_r)$,
\[
\int f\,d\tilde{\mu}=\frac{\int_X\!\Bigl(\int_0^{r(x)}f(x,s)\,d\mathrm{Leb}(s)\Bigr)\,d\mu(x)}{\int_X r(x)\,d\mu(x)}.
\]
In particular,
\[
\mathcal{M}^e(X_r,\phi)=\{\tilde{\mu}:\mu\in\mathcal{M}^e(X,T)\}
\]
is countable, because $(X,T)$ has only countably many ergodic measures.

For every $\mu\in\mathcal{M}(X,T)$ the Abramov formula \cite{A} gives
\[
h_{\tilde{\mu}}(\varphi_1)=\frac{h_\mu(T)}{\int_X r(x)\,d\mu(x)}=0,
\]
the last equality holding since $(X,T)$ has zero entropy. By the variational principle, $h_{\mathrm{top}}(X_r,\varphi_1)=0$.

Next we show that $(X_r,\varphi_1)$ has almost countable spectrum. Fix $\mu\in\mathcal{M}^e(X,T)$ and let
\(\tilde{\mu}=\int_\Omega\tilde{\mu}_\omega\,d\xi(\omega)\)
be the ergodic decomposition of $\tilde{\mu}$ with respect to $\varphi_1$. Then
\(\tilde{\mu}=\int_\Omega\!\Bigl(\int_0^1(\varphi_t)_*\tilde{\mu}_\omega\,dt\Bigr)\,d\xi(\omega)\).
Since $\tilde{\mu}$ is $\Phi$-ergodic and each measure $\int_0^1(\varphi_t)_*\tilde{\mu}_\omega\,dt$ is $\Phi$-ergodic, we have $\int_0^1(\varphi_t)_*\tilde{\mu}_\omega\,dt=\tilde{\mu}$ for $\xi$-a.e.\ $\omega\in\Omega$. Hence we can choose a $\varphi_1$-ergodic measure $m_{\tilde{\mu}}$ (in fact $m_{\tilde{\mu}}=\tilde{\mu}_\omega$ for some $\omega$) such that $\int_0^1(\varphi_t)_*m_{\tilde{\mu}}\,dt=\tilde{\mu}$.

Put
\[
C:=\bigcup_{\mu\in\mathcal{M}^e(X,T)}\mathrm{Spec}(X_r,\mathcal{X}_r,m_{\tilde{\mu}},\varphi_1),
\]
where $\mathcal{X}_r$ is the Borel $\sigma$-algebra of $X_r$. Then $C$ is a countable subset of $\mathbb{T}$, because $\mathcal{M}^e(X,T)$ is countable and each set $\mathrm{Spec}(X_r,\mathcal{X}_r,m_{\tilde{\mu}},\varphi_1)$ is countable.

We now show that for every $\nu\in\mathcal{M}(X_r,\varphi_1)$,
\[
\mathrm{Spec}(X_r,\mathcal{X}_r,\nu,\varphi_1)\subset C,
\]
which clearly implies that $(X_r,\varphi_1)$ has almost countable spectrum.

Indeed, for any $\nu\in\mathcal{M}(X_r,\varphi_1)$ the measure $\int_0^1(\varphi_t)_*\nu\,dt$ is $\Phi$-ergodic, so there exists $\mu_\nu\in\mathcal{M}^e(X,T)$ such that $\tilde{\mu}_\nu=\int_0^1(\varphi_t)_*\nu\,dt$. Thus $\int_0^1(\varphi_t)_*\nu\,dt$ is the ergodic decomposition of $\tilde{\mu}_\nu$ and coincides with $\int_0^1(\varphi_t)_*m_{\tilde{\mu}_\nu}\,dt=\tilde{\mu}_\nu$. Consequently there is an $s\in[0,1)$ such that $\nu=(\varphi_s)_*m_{\tilde{\mu}_\nu}$. Since $\varphi_s\circ\varphi_1=\varphi_1\circ\varphi_s$, we obtain
\begin{align*}
\mathrm{Spec}(X_r,\mathcal{X}_r,\nu,\varphi_1)
&=\mathrm{Spec}(X_r,\mathcal{X}_r,(\varphi_s)_*m_{\tilde{\mu}_\nu},\varphi_1)\\
&=\mathrm{Spec}(X_r,\mathcal{X}_r,m_{\tilde{\mu}_\nu},\varphi_1)\subseteq C.
\end{align*}
Therefore $\mathrm{Spec}(X_r,\mathcal{X}_r,\nu,\varphi_1)\subseteq C$, completing the proof of Proposition~\ref{SZ-susp}.
\end{proof}

\subsection{Finite maximal pattern entropy} The aim of this subsection is to prove the following Proposition \ref{lem-mb-mpe-zc}: if a TDS which each invariant measure has  finite maximal pattern entropy, then it has zero entropy and almost countable spectrum.  Thus Theorem \ref{LSC-bounded-m-MPE}
 is a direct corollary of Proposition  \ref{lem-mb-mpe-zc} and Theorem \ref{thm-A}.

To prove Proposition~\ref{lem-mb-mpe-zc}, we need several additional lemmas. The first one asserts that for systems with discrete spectrum, the spectrum of almost every ergodic component is contained in the spectrum of the system.
\begin{lem}\label{lem-sp-rho=rhopmega}
	Let $(X,\mathcal{X},\mu,T)$ be a discrete spectrum system. If $\mu=\int_{\Omega} \mu_\omega d\xi(\omega)$ is the ergodic decomposition of $\mu$, then
 $\text{Spec}(X,\mathcal{X},\mu_\omega,T)\subset \text{Spec}(X,\mathcal{X},\mu,T)$, for $\xi$-a.e. $\omega$ in $\Omega$.
 \end{lem}
\begin{proof} Since $(X,\mathcal{X},\mu,T)$ is discrete spectrum,
$$L^2(X,\mathcal{X},\mu)=L^2(X,\mathcal{K}(X,\mathcal{X},\mu,T),\mu).$$
 Thus by Lemma~\ref{dec-rat-part}~(1), there exists a countable orthonormal  basis $\mathcal{F}:=\{f_i\}_{i\in I}$ of $L^2(X,\mathcal{X},\mu)$ such that  each $f_i$ is an eigenfunction  associated to some eigenvalue $\lambda_i$. Especially, since $\mathcal{F}$ is countable, it follows that there exists a full measure subset $\Omega'$ of $\Omega$ such that for $\omega\in \Omega'$, each $f_i$ is also an eigenfunction of $(X,\mathcal{X},\mu_\omega,T)$  associated to $\lambda_i$.
	
	We claim that for every $h\in C(X)$, for almost all $\omega\in \Omega'$, $h$ belong to  the closure of span of $\mathcal{F}$ in $L^2(\mu_\omega)$. Fix $h\in C(X)$, since $(X,\mathcal{X},\mu,T)$ has discrete spectrum for each $n\in \mathbb{N}$ there exists $K_n\in\N$ and $p_1^{(n)},\cdots,p_{K_n}^{(n)}\in\mathbb{C}$ such that
	$$\|f-\sum_{j=1}^{K_n}p_i^{(n)}f_i\|_{L^2(\mu)}\le 1/n^2.$$
	This implies
		$$\int_{\Omega}\|f-\sum_{j=1}^{K_n}p_i^{(n)}f_i\|_{L^2(\mu_\omega)}^2d\xi(\omega)\le 1/n^4.$$
		Put
		$$\Omega_n:=\{\omega\in \Omega: \|f-\sum_{j=1}^{K_n}p_i^{(n)}f_i\|_{L^2(\mu_\omega)}\le 1/n \}.$$ Then, by Markov's inequality (see for example (5.30) in \cite{B95}) $\xi(\Omega_n)\ge 1-1/n^2$.
		It is easy to see that for $\omega\in \left(\cap_{N\in\N}\cup_{n\ge N}\Omega_n\right)\cap \Omega'$, $h$ belong to  the closure of span of $\mathcal{F}$ in $L^2(\mu_\omega)$. We prove the claim since $ \left(\cap_{N\in\N}\cup_{n\ge N}\Omega_n\right)\cap \Omega'$ has full measure and every element in $\mathcal{F}$ is an eigenfunction of $(X,\mathcal{X}, \mu_\omega,T)$ for  $\omega\in \Omega'$.
		
		Now, since $C(X)$ is separable, we can find a countable dense subset $\mathcal{H}$ of $C(X)$. Together with the claim above, for  almost all $\omega\in \Omega'$, $\mathcal{H}$ is in the closure of span of $\mathcal{F}$ in $L^2(\mu_\omega)$. Since for every $\omega\in\Omega$, $\mathcal{H}$ is dense in $L^2(\mu_\omega)$, it follows that for  almost all $\omega\in \Omega'$, $L^2(\mu_\omega)$ is the closure of span of $\mathcal{F}$. Since every $f\in \mathcal{F}$ is an eigenfunction  associated to some spectrum in $\text{Spec}(X,\mathcal{X},\mu,T)$, the lemma follows.
	\end{proof}

Let \(X\) be a compact metric space and \(\mu\in\mathcal{M}(X)\). \(\mu\) is \emph{atomic} if there exists a countable set \(E\subset X\) such that \(\mu(E)=1\). \(\mu\) is \emph{atomless} if \(\mu(\{x\})=0\) for every \(x\in X\). For any \(\mu\in\mathcal{M}(X)\) define \(E=\{x\in X:\mu(\{x\})>0\}\).
Then \(E\) is a Borel set, at most countable, and the restriction \(\mu|_E\) is atomic while \(\mu|_{X\setminus E}\) is atomless.
By Lebesgue's decomposition theorem,
\(\mu=\mu(E)\cdot\mu|_E+\bigl(1-\mu(E)\bigr)\cdot\mu|_{X\setminus E}\),
where the two restricted measures are normalized to probabilities.

The following result should be regarded as standard; for completeness we include a proof.
\begin{lem}\label{lem-1}
Let $\mu\in\mathcal{M}(X)$ and let $\{\alpha_i\}_{i\in\mathbb{N}}$ be a sequence of measurable partitions of $X$ whose diameters tend to zero. If $\mu$ is not an  atomic measure, then $\lim_{i\to\infty}H_\mu(\alpha_i)=+\infty$.
\end{lem}
\begin{proof}
Put $E=\{x:\mu(\{x\})>0\}$. Since $\mu$ is not an atomic measure,  $\mu(E)<1$. By the concavity of  $-t\log t$ on $(0,+\infty)$ (see \cite[Lemma~3.2(i)]{LW77}),
\[
H_\mu(\alpha_i)\ge(1-\mu(E))\,H_{\mu|_{X\setminus E}}(\alpha_i).
\]
Since $\max_{A\in\alpha_i}\mu|_{X\setminus E}(A)\to 0$ as $i\rightarrow +\infty$, we have
$$H_{\mu|_{X\setminus E}}(\alpha_i)\ge-\log(\max_{A\in\alpha_i}\mu|_{X\setminus E}(A))\to +\infty$$ as $i\rightarrow +\infty$, and the lemma follows.
\end{proof}
%

We recall that an almost everywhere finite-to-one factor map between two measure preserving systems were introduced earlier in Definition~\ref{de-fi-one}.
\begin{lem}\label{lem-2}Let $(X,\mathcal{X},\mu,T)$ be a system with finite maximal pattern entropy. Let $\pi: (X,\mathcal{X},\mu,T)\to (Y,\mathcal{Y},\nu,R)$ be a factor map with $\pi^{-1}(\mathcal{Y})=\mathcal{K}(X,\mathcal{X},\mu,T)$  modulo $\mu$-null sets. If $\mu=\int_\Omega \mu_\omega d\xi(\omega)$ is the ergodic decomposition of $\mu$, then for $\xi$-a.e. $\omega\in\Omega$, the factor map $\pi:( X,\mathcal{X},\mu_\omega,T)\to (Y,\mathcal{Y},\pi_*(\mu_\omega),R)$ is almost everywhere finite-to-one.
\end{lem}
\begin{proof}For \(\xi\)-a.e.\ \(\omega\in\Omega\) we write \(\nu_\omega=\pi_*(\mu_\omega)\). There exists a Borel set \(\Omega'\subset\Omega\) with \(\xi(\Omega')=1\) such that for every \(\omega\in\Omega'\),
\[
\pi\colon(X,\mathcal{X},\mu_\omega,T)\to(Y,\mathcal{Y},\nu_\omega,R)
\]
is a factor map between two ergodic systems. Let
\(
\mu_\omega=\int_Y\mu_{\omega,y}\,d\nu_\omega(y)
\)
be the disintegration of \(\mu_\omega\) with respect to \((Y,\mathcal{Y},\nu_\omega,R)\). Fix a sequence \(\{\alpha_i\}_{i\in\mathbb{N}}\) of measurable partitions of \(X\) with \(\operatorname{diam}(\alpha_i)\to 0\) and \(\alpha_1\preceq\alpha_2\preceq\alpha_3\preceq\cdots\). For brevity write \(\mathcal{K}(\mu)=\mathcal{K}(X,\mathcal{X},\mu,T)\). Theorem~\ref{thm-se-enytopy} gives
\begin{align}\label{eq-1234-0}
h_\mu^*(T)\ge\lim_{i\to\infty}h_\mu^*(T,\alpha_i)=\lim_{i\to\infty}H_\mu(\alpha_i|\mathcal{K}(\mu))=\lim_{i\to\infty}H_\mu(\alpha_i|\pi^{-1}(\mathcal{Y})).
\end{align}
By Jensen's inequality and the concavity of \(H_{\cdot}(\alpha_i|\pi^{-1}(\mathcal{Y}))\) on \(\mathcal{M}(X)\) (see \cite[Lemma~3.3(1)]{HYZ06}),
\begin{equation}\label{Jesen-eq-1}
H_\mu(\alpha_i|\pi^{-1}(\mathcal{Y}))\ge\int_\Omega H_{\mu_\omega}(\alpha_i|\pi^{-1}(\mathcal{Y}))\,d\xi(\omega).
\end{equation}
Since $\alpha_1\preceq\alpha_2\preceq\cdots$, for $\xi$-a.e.\ $\omega\in\Omega$ the sequence $\{H_{\mu_\omega}(\alpha_i|\pi^{-1}(\mathcal{Y}))\}_{i\in\mathbb{N}}$ is non-negative and non-decreasing in $i$. From \eqref{eq-1234-0} and \eqref{Jesen-eq-1} we obtain
\begin{align}\label{eq-1234}
h_\mu^*(T)\ge\int_\Omega\lim_{i\to\infty}H_{\mu_\omega}(\alpha_i|\pi^{-1}(\mathcal{Y}))\,d\xi(\omega).
\end{align}
As $h_\mu^*(T)<\infty$, we have $\lim_{i\to\infty}H_{\mu_\omega}(\alpha_i|\pi^{-1}(\mathcal{Y}))<\infty$ for $\xi$-a.e.\ $\omega\in\Omega$.

Set $\Omega_c=\{\omega\in\Omega':\lim_{i\to\infty}H_{\mu_\omega}(\alpha_i|\pi^{-1}(\mathcal{Y}))<\infty\}$. Then $\Omega_c$ is measurable and $\xi(\Omega_c)=1$. We shall show that for every $\omega\in\Omega_c$ the factor map
\[
\pi\colon(X,\mathcal{X},\mu_\omega,T)\to(Y,\mathcal{Y},\nu_\omega,R)
\]
is  almost everywhere finite-to-one.

Fix $\omega\in\Omega_c$. Since
\[
H_{\mu_\omega}(\alpha_i|\pi^{-1}(\mathcal{Y}))=\int_Y H_{\mu_{\omega,y}}(\alpha_i)\,d\nu_\omega(y)\text{ for each }i\in\mathbb{N},
\]
and the sequence $\{H_{\mu_{\omega,y}}(\alpha_i)\}_{i\in\mathbb{N}}$ is non-negative and non-decreasing in $i$ for $\nu_\omega$-a.e.\ $y\in Y$, we get
\[
\infty>\lim_{i\to\infty}H_{\mu_\omega}(\alpha_i|\pi^{-1}(\mathcal{Y}))\ge\int_Y\lim_{i\to\infty}H_{\mu_{\omega,y}}(\alpha_i)\,d\nu_\omega(y).
\]
Hence $\lim_{i\to\infty}H_{\mu_{\omega,y}}(\alpha_i)<\infty$ for $\nu_\omega$-a.e.\ $y\in Y$, so Lemma~\ref{lem-1} implies that $\mu_{\omega,y}$ is atomic for $\nu_\omega$-a.e.\ $y\in Y$.

By the Rohlin skew-product theorem (see e.g.\ \cite[Theorem~3.18]{G03}), a factor map between two ergodic systems is either almost everywhere finite-to-one or, for almost every fibre, the conditional measure is atomless. Since $\mu_{\omega,y}$ is atomic for $\nu_\omega$-a.e.\ $y\in Y$, the map
\(
\pi\colon(X,\mathcal{X},\mu_\omega,T)\to(Y,\mathcal{Y},\nu_\omega,R)
\)
is almost everywhere finite-to-one, completing the proof of Lemma~\ref{lem-2}.
\end{proof}

For  $C\subset \mathbb{T}$,  let
	$$C^\mathbb{Q}:=\left\{e(\sum_{i=1}^k q_it_{i}+q): k\in\mathbb{N}, q\in\mathbb{Q} \text{ and } q_i\in \mathbb{Q}, e(t_i)\in C \text{ for }i=1,\cdots,k\right\}.$$
	It is easy to see that if $C$ is a countable subset of $\mathbb{T}$, then so is $C^\mathbb{Q}$.
	\begin{lem}\label{lem-23}Let $\pi: (X,\mathcal{X},\mu,T)\to (Y,\mathcal{Y},\nu,R)$ be a factor map between two ergodic systems. If $\pi$ is almost everywhere finite-to-one, then $$\text{Spec}_{irr}(X,\mathcal{X},\mu,T)\subset \text{Spec}_{irr}(Y,\mathcal{Y},\nu,R)^\mathbb{Q}.$$
	\end{lem}
\begin{proof} Let  $\lambda\in\text{Spec}_{irr}(X,\mathcal{X},\mu,T)$. Assume $f\in L^2(\mu)$ satisfies $U_Tf=\lambda f$. Due to ergodicity, we can further assume that $|f(x)|=1$ for $\mu$-a.e. $x\in X$. Define a factor map
 $$\widetilde{\pi}: (X,\mathcal{X},\mu,T)\rightarrow (\mathbb{T}\times Y,\mathcal{B}_\mathbb{T}\times\mathcal{Y}, \widetilde{\pi}_*(\mu), R_\lambda\times R),$$
  such that $\widetilde{\pi}(x)=(f(x),\pi(x))$ for $x\in X$,
    where $R_\lambda: \mathbb{T}\rightarrow \mathbb{T}$ is the rotation: $R(z)=\lambda z$ for $z\in\mathbb{T}$ and $\mathcal{B}_{\mathbb{T}}$ is the Borel-$\sigma$ algebra of $\mathbb{T}$. Let $m_\mathbb{T}$ be the Haar measure on $\mathbb{T}$. Since $\lambda $ is an  irrational eigenvalue with respect to the eigenfunction $f$, it follows that $\widetilde{\pi}_*(\mu)$ is a joining of $(\mathbb{T},\mathcal{B}_{\mathbb{T}},m_\mathbb{T},R_{\lambda})$ and $(Y,\mathcal{Y},\nu,R)$.

    Note that $\text{Spec}(\mathbb{T},\mathcal{B}_\mathbb{T}, m_\mathbb{T},R_\lambda)=\{\lambda^n: n\in \mathbb{Z}\}$. If $\lambda\notin \text{Spec}_{irr}(Y,\mathcal{Y},\nu,R)^\mathbb{Q}$, the rotation $(\mathbb{T},\mathcal{B}_\mathbb{T}, m_\mathbb{T},R_\lambda)$ (which is an ergodic $1$-step nilsystem) and the ergodic system $(Y,\mathcal{Y},\nu,R)$ have disjoint spectrum different than $1$.  Due to Lemma \ref{lemma-eig-dis} (ii), they are disjoint, which implies that $ \widetilde{\pi}_*(\mu)=m_\mathbb{T}\times\nu$. Hence, the factor map $\pi_Y:(\mathbb{T}\times Y, \mathcal{B}_\mathbb{T}\times\mathcal{Y}, \widetilde{\pi}_*(\mu), R_\lambda\times R)\rightarrow (Y,\mathcal{Y},\nu,R)$ is not almost everywhere finite-to-one, where $\pi_Y: \mathbb{T}\times Y\rightarrow Y$ is the coordinate projection. However, by hypothesis $\pi=\pi_Y\circ\widetilde{\pi}$ is almost everywhere finite-to-one, it follows that $\pi_Y$ is almost everywhere finite to one, a contradiction. Thus, $\lambda\in \text{Spec}_{irr}(Y,\mathcal{Y},\nu,R)^\mathbb{Q}$, which proves Lemma \ref{lem-23} by the arbitrariness of irrational eigenvalue $\lambda$.
	\end{proof}

\begin{prop}\label{lem-mb-mpe-zc}
Let $(X,T)$ be a TDS. If every invariant probability measure of $(X,T)$ has finite maximal pattern entropy, then $(X,T)$ has zero entropy and almost countable spectrum.
\end{prop}

\begin{proof}
Let $\mu\in\mathcal{M}(X,T)$. By assumption, the maximal pattern entropy satisfies $h_\mu^*(T)<+\infty$; hence $h_\mu(T)=0$ by Theorem~\ref{MPE-p}~(4). It is well known (see, for example, \cite[Theorem~5.15]{F81} or \cite[Theorem~6.5]{EW11}) that there exists a factor map $\pi\colon (X,\mathcal{X},\mu,T)\to (Y,\mathcal{Y},\nu,R)$ such that $\pi^{-1}(\mathcal{Y})=\mathcal{K}(X,\mathcal{X},\mu,T)\mod\mu$. Thus $(Y,\mathcal{Y},\nu,R)$ has discrete spectrum.

Let $\mu=\int_{\Omega}\mu_\omega\,d\xi(\omega)$ be the ergodic decomposition of $\mu$. Then
\[
\nu=\pi_*(\mu)=\int_{\Omega}\pi_*(\mu_\omega)\,d\xi(\omega)
\]
is the ergodic decomposition of $\nu$. By Lemma~\ref{lem-sp-rho=rhopmega},
\begin{equation}\label{spect-relation-dec-1}
\text{Spec}_{irr}(Y,\mathcal{Y},\pi_*(\mu_\omega),T)\subset \text{Spec}_{irr}(Y,\mathcal{Y},\nu,T)
\end{equation}
for $\xi$-a.e.\ $\omega\in \Omega$.

By Lemma~\ref{lem-2}, the factor map $\pi\colon (X,\mathcal{X},\mu_\omega,T)\to (Y,\mathcal{Y},\pi_*(\mu_\omega),R)$ is almost everywhere finite-to-one for $\xi$-a.e.\ $\omega\in\Omega$. Lemma~\ref{lem-23} and \eqref{spect-relation-dec-1} then give, for $\xi$-a.e.\ $\omega\in\Omega$,
\[
\text{Spec}_{irr}(X,\mathcal{X},\mu_\omega,T)\subset\text{Spec}_{irr}(Y,\mathcal{Y},\pi_*(\mu_\omega),T)^{\mathbb{Q}}\subset \text{Spec}_{irr}(Y,\mathcal{Y},\nu,T)^{\mathbb{Q}}.
\]
Since $\text{Spec}_{irr}(Y,\mathcal{Y},\nu,T)$ is countable, so is $\text{Spec}_{irr}(Y,\mathcal{Y},\nu,T)^{\mathbb{Q}}$; hence $(X,\mathcal{X},\mu,T)$ has almost countable spectrum. As $\mu\in\mathcal{M}(X,T)$ was arbitrary, $(X,T)$ has almost countable spectrum and zero entropy (by the variational principle). This completes the proof of Proposition~\ref{lem-mb-mpe-zc}.
\end{proof}

\noindent{\bf Acknowledgment.}
This work is supported by the National Key R\&D Program of China (Nos. 2024YFA1013602, 2024YFA1013600) and the National Natural Science Foundation of China (Nos.12031019,12371197,12426201). The authors thank Professor Song Shao for bringing Theorem 1.1 of Rudolph's article \cite{Ru} to their attention.

	\end{document}